\documentclass{amsart}

\usepackage{ae,aecompl}
\usepackage{esint}
\usepackage{graphicx}
\usepackage[all,cmtip]{xy}
\usepackage{amsmath, amscd}
\usepackage{booktabs}
\usepackage{amsthm}
\usepackage{amssymb}
\usepackage{amsfonts}
\usepackage{qsymbols}
\usepackage{latexsym}
\usepackage{mathrsfs}
\usepackage{cite}
\usepackage{color}
\usepackage{url}
\usepackage{enumerate}
\usepackage{undertilde}
\usepackage{enumitem}
\usepackage{verbatim}
\usepackage[draft=false, colorlinks=true]{hyperref}

\allowdisplaybreaks

\newcommand {\abs}[1]{\lvert#1\rvert}

\newcommand {\C}{{\mathbb C}}

\newcommand {\D}{D}

\newcommand {\ud}{\mathrm{d}}
\newcommand {\ue}{\mathrm{e}}

\newcommand {\veps}{\varepsilon}

\newcommand {\Ell}{L}

\newcommand {\F}{{\mathcal{F}}}

\newcommand {\HT}{\mathcal{H}}

\newcommand {\ui}{\mathrm{i}}
\newcommand {\I}{{I}}

\newcommand {\rb}{\rangle}
\newcommand {\lb}{{\langle}}
\newcommand {\La}{{\mathcal{L}}}

\newcommand {\Ma}{{\mathcal{M}}}
\newcommand {\N}{{{\mathbb N}}}

\newcommand {\norm}[1]{\left\|#1\right\|}

\newcommand {\ph}{{\varphi}}

\newcommand {\R}{{\mathbb R}}

\newcommand {\Rn}{{\mathbb{R}^{n}}}

\newcommand {\St}{{\mathrm{St}}}
\newcommand {\Sw}{\mathcal{S}}

\newcommand {\w}{{\omega}}

\newcommand {\Z}{{{\mathbb Z}}}

\newcommand {\vanish}[1]{\relax}

\newcommand{\wh}{\widehat}
\newcommand{\wt}{\widetilde}

\DeclareFontFamily{U}{mathx}{\hyphenchar\font45}
\DeclareFontShape{U}{mathx}{m}{n}{
      <5> <6> <7> <8> <9> <10>
      <10.95> <12> <14.4> <17.28> <20.74> <24.88>
      mathx10
      }{}
\DeclareSymbolFont{mathx}{U}{mathx}{m}{n}
\DeclareFontSubstitution{U}{mathx}{m}{n}
\DeclareMathAccent{\widecheck}{0}{mathx}{"71}

\DeclareMathOperator{\Real}{Re}
\DeclareMathOperator{\Imag}{Im}

\newtheorem{theorem}{Theorem}[section]
\newtheorem{lemma}[theorem]{Lemma}
\newtheorem{proposition}[theorem]{Proposition}
\newtheorem{corollary}[theorem]{Corollary}

\theoremstyle{definition}
\newtheorem{definition}[theorem]{Definition}

\DeclareFontFamily{U}{mathx}{\hyphenchar\font45}
\DeclareFontShape{U}{mathx}{m}{n}{
      <5> <6> <7> <8> <9> <10>
      <10.95> <12> <14.4> <17.28> <20.74> <24.88>
      mathx10
      }{}
\DeclareSymbolFont{mathx}{U}{mathx}{m}{n}
\DeclareFontSubstitution{U}{mathx}{m}{n}
\DeclareMathAccent{\widecheck}{0}{mathx}{"71}

\numberwithin{equation}{section}
\protected\def\ignorethis#1\endignorethis{}
\let\endignorethis\relax

\usepackage{crossreftools}

\makeatletter
\newcommand{\optionaldesc}[2]{%
  \phantomsection
  #1\protected@edef\@currentlabel{#1}\label{#2}%
}
\makeatother

\title[$(L^{p},L^{q})$ Fourier multipliers and stability theory]{Operator-valued $(L^{p},L^{q})$ Fourier multipliers\\and stability theory for evolution equations}

\author{Jan Rozendaal}
\address{Institute of Mathematics, Polish Academy of Sciences\\
ul.~\'{S}niadeckich 8\\
00-656 Warsaw\\
Poland}
\email{jrozendaal@impan.pl}



\thanks{The research leading to these results has received funding from the Norwegian Financial Mechanism 2014-2021, grant 2020/37/K/ST1/02765.}

\begin{document}

\begin{abstract}
We give an overview of some recent results on operator-valued $(L^{p},L^{q})$ Fourier multipliers and stability theory for evolution equations. The aim is to provide a relatively nontechnical introduction to the underlying ideas, emphasizing the connection between the two areas. We also indicate how operator-valued $(L^{p},L^{q})$ Fourier multipliers can be applied to functional calculus theory.  
\end{abstract}

\maketitle

\section{Introduction}\label{sec:introduction}

One of the main objectives in the theory of partial differential equations is to determine the asymptotic behavior of solutions to a given partial differential equation. For those linear partial differential equations which can be conveniently analyzed by reformulating them as evolution equations, it is well known that the long-term behavior of solutions is related to properties of the generator of the solution semigroup. In this article we indicate how recent results on vector-valued harmonic analysis yield a new connection between properties of the semigroup generator, the geometry of the underlying Banach space, and asymptotic behavior of solutions.

\subsection{Setting}

Consider the abstract Cauchy problem
\begin{equation}\label{eq:evoleq}
\begin{aligned}
\dot{u}(t)&=Au(t),\quad t\geq0,\\
u(0)&=x,
\end{aligned}
\end{equation}
where $u:[0,\infty)\to X$ for some Banach space $X$, $A$ is a closed operator on $X$, and $x\in X$. Then \eqref{eq:evoleq} has a unique mild solution for every $x\in X$, and the solution depends continuously on $x$, if and only if $A$ generates a strongly continuous semigroup $(T(t))_{t\geq0}\subseteq\La(X)$ on $X$. In this case the unique solution $u$ to \eqref{eq:evoleq} is given by $u(t)=T(t)x$ for all $t\geq0$. Moreover, if $x\in D(A)$ then $u\in C^{1}([0,\infty);X)$, and \eqref{eq:evoleq} holds pointwise. 

Throughout, we make the assumption of existence and uniqueness of solutions to \eqref{eq:evoleq}, and continuous dependence on initial data. Then, to determine the asymptotic behavior of solutions to \eqref{eq:evoleq}, one has to analyze the long-term behavior of the orbits $t\mapsto T(t)x$. To approach the latter problem one can in turn examine spectral properties of the generator $A$, thereby applying the classical paradigm for ordinary differential equations to evolution equations. 

Indeed, there exist $\w'\in\R$ and $C\geq1$ such that $\|T(t)\|_{\La(X)}\leq Ce^{\w' t}$ for all $t\geq0$. Without loss of generality, suppose that $\w'=0$. Then the resolvent $R(\w+i\xi,A)$ exists for all $\w>0$ and $\xi\in\R$, and
\begin{equation}\label{eq:Laplace}
R(\w+i\xi,A)x=\int_{0}^{\infty}e^{-t(\w+i\xi)}T(t)x\,\ud t
\end{equation}
for all $x\in X$ (see \cite[Theorem II.1.10]{Engel-Nagel00}). One can invert this Laplace transform:
\begin{equation}\label{eq:inverseFourier}
e^{-\w t}T(t)x=\frac{1}{2\pi}\int_{\R}e^{it\xi}R(\w+i\xi,A)x\,\ud \xi
\end{equation}
for $t\geq0$, where the integral converges absolutely for $x$ in suitable dense subspaces of $X$, such as the fractional domain $D((-A)^{\alpha})$ for $\alpha>1$. Loosely speaking, for a general $x\in X$ one can then use information about $R(\w+i\xi,A)x$ to bound the integral in \eqref{eq:inverseFourier} and obtain a $C'=C'_{x}\geq0$ such that
\begin{equation}\label{eq:asymptotic}
\|T(t)x\|_{X}\leq C'e^{\w t},\quad t\geq0.
\end{equation}
Of course, we already assumed from the start that $\|T(t)\|_{\La(X)}\leq C$, so this is of little use for $\w\geq 0$. However, if $R(\w+i\xi,A)$ exists for some $\w<0$ and all $\xi\in\R$, then one can still use \eqref{eq:inverseFourier} to relate information about $R(\w+i\xi,A)$ to bounds as in \eqref{eq:asymptotic}. Since $\w<0$, one thus obtains decay estimates for the orbit $t\mapsto T(t)x$. This approach to the asymptotic behavior of solutions to evolution equations is useful in practice, because the resolvent is typically more accessible than the semigroup itself.

Let us make this heuristic precise in a special case. Suppose that $i\R\subseteq\rho(A)$ and that
\begin{equation}\label{eq:supremumbound}
\sup_{\xi\in\R}\|R(i\xi,A)\|_{\La(X)}<\infty.
\end{equation} 
Then a Neumann series argument yields an $\w<0$ such that $\w+i\R\subseteq\rho(A)$ and
\begin{equation}\label{eq:supremumbound2}
\sup_{\xi\in\R}\|R(\w+i\xi,A)\|_{\La(X)}<\infty.
\end{equation}
In turn, one can use \eqref{eq:supremumbound2} to obtain
\begin{equation}\label{eq:integrability}
\int_{\R}\|R(\w+i\xi,A)x\|_{X}\,\ud \xi<\infty
\end{equation}
for all $x\in D((-A)^{\alpha})$ for $\alpha>1$. And then \eqref{eq:inverseFourier} indeed implies \eqref{eq:asymptotic} for such $x$.

This argument works on general Banach spaces. On the other hand, if $X$ is a Hilbert space, then Plancherel's theorem, \eqref{eq:inverseFourier} and \eqref{eq:supremumbound2} can be combined to yield \eqref{eq:asymptotic} for all $x\in X$. This is the so-called Gearhart--Pr\"{u}ss theorem (see~\cite[Theorem V.1.11]{Engel-Nagel00}). The difference between obtaining \eqref{eq:asymptotic} for all $x\in X$, or only for $x\in D((-A)^{\alpha})$ for some $\alpha>0$, is relevant in practice. Indeed, for most applications to partial differential equations, $A$ is a differential operator, and then elements of the fractional domains $D((-A)^{\alpha})$ get progressively ``smoother" as $\alpha$ grows. Hence, on Hilbert spaces, \eqref{eq:supremumbound} yields \eqref{eq:asymptotic} for arbitrarily rough initial data, while on general Banach spaces we only obtain \eqref{eq:asymptotic} for sufficiently smooth initial data.

Next, it is natural to attempt to relax the assumptions that we have made. The work discussed in this article is motivated by three loosely formulated problems which arise in this manner:
\begin{enumerate}[label=(P\arabic*)]
\item\label{it:problem1} Obtain \eqref{eq:asymptotic} for $x\in D((-A)^{\alpha})$, for some $\w<0$ and $0<\alpha<1$, under the assumption that the underlying Banach space has more structure than a general Banach space but less than a Hilbert space;
\item\label{it:problem2} Determine the asymptotic behavior of semigroup orbits under the assumption that there exists an $M:[0,\infty)\to[0,\infty)$ such that $M(\lambda)\to\infty$ as $\lambda\to\infty$, and such that $\|R(i\xi,A)\|_{\La(X)}\leq M(|\xi|)$ for $\xi\in\R$. 
\item\label{it:problem3} Determine the asymptotic behavior of semigroup orbits under the assumption that there exists a $g:(0,\infty)\to(0,\infty)$ such that $g(t)\to\infty$ as $t\to\infty$, and such that $\|R(\lambda,A)\|_{\La(X)}\leq g(1/\Real(\lambda))$ whenever $\Real(\lambda)>0$.
\end{enumerate}
The first setting arises, for example, in time--frequency analysis and signal processing \cite{Grochenig01,Chaichenets18}, where one encounters initial data which decays slower at infinity than $L^{2}$ functions do. Moreover, the $L^{p}$ norm captures additional information for $p\neq 2$, such as concentration and dispersion effects. On the other hand, \ref{it:problem2} is of interest, even on Hilbert spaces, when considering more refined asymptotic behavior than \eqref{eq:asymptotic}. Such behavior in turn arises in applications to damped wave equations, where \eqref{eq:supremumbound} often does not hold. Finally, Problem \ref{it:problem3} arises when one relaxes standard resolvent conditions which guarantee that a semigroup is uniformly bounded.

\subsection{Contents of this article}

We will discuss the following recent developments, which will be described in more detail in the next section:
\begin{enumerate}[label=(D\arabic*)]
\item\label{it:develop1} The development of a theory of operator-valued $(L^{p},L^{q})$ Fourier multipliers for $p\neq q$;
\item\label{it:develop2} An application of the theory of $(L^{p},L^{q})$ Fourier multipliers to Problem \ref{it:problem1}, indicating how such multipliers can be used to obtain concrete decay rates on non-Hilbertian Banach spaces;
\item\label{it:develop3} A connection between the theory of $(L^{p},L^{q})$ Fourier multipliers and polynomial decay behavior in Problem \ref{it:problem2}, yielding concrete decay rates;
\item\label{it:develop4} An application of the theory of $(L^{p},L^{q})$ Fourier multipliers to Problem \ref{it:problem3}, yielding concrete growth rates;
\item\label{it:develop5} A characterization of those functions $M$ in \ref{it:problem2} for which there exist $c,C>0$ such that $\|T(t)A^{-1}\|_{\La(X)}\leq \frac{C}{M^{-1}(ct)}$ for sufficiently large $t$, whenever $(T(t))_{t\geq0}$ is a uniformly bounded semigroup on a Hilbert space $X$.
\end{enumerate}
We also give an application of the theory of $(L^{p},L^{q})$ Fourier multipliers to functional calculus theory. These results have applications to stability theory of a different nature, namely the stability of numerical approximation schemes for solutions to evolution equations.

Development \ref{it:develop1} is based on the articles \cite{Rozendaal-Veraar18,Rozendaal-Veraar17b}, and \ref{it:develop2} and \ref{it:develop3} rely on \cite{Rozendaal-Veraar18a}. Development \ref{it:develop4} is contained in \cite{Rozendaal-Veraar18b}, and \ref{it:develop5} is based on work in \cite{Batty-Duyckaerts08,Chill-Seifert16,Borichev-Tomilov10,BaChTo16,RoSeSt19}. The application of $(L^{p},L^{q})$ Fourier multipliers to functional calculus theory is taken from \cite{Rozendaal19}. 
It is important to stress that the work described in this article is a continuation of developments by many authors, and we will attempt to indicate the essential role that this work played in the material discussed here. Moreover, we focus on the ideas underlying the proofs, referring to the relevant articles for technical details.

This article is organized as follows. In Section \ref{sec:previous} we discuss Problems \ref{it:problem1}, \ref{it:problem2} and \ref{it:problem3} in more detail. Sections \ref{sec:multipliers}, \ref{sec:exponential}, \ref{sec:polynomial}, \ref{sec:growth} and \ref{sec:refined} deal with Developments \ref{it:develop1}, \ref{it:develop2}, \ref{it:develop3}, \ref{it:develop4} and \ref{it:develop5}, respectively. In Appendix \ref{sec:calctypecotype} we discuss an application of $(L^{p},L^{q})$ Fourier multipliers to functional calculus theory. Finally, in Appendix \ref{sec:open} we formulate some open problems.

\subsection{Aim of this article}

Our goal in writing this article has not been to give a reasonably complete overview of the asymptotic theory of evolution equations. Many excellent such texts exist, both classical (see e.g.~\cite{ArBaHiNe11,vanNeerven96b,Engel-Nagel00}) and more recent ones \cite{ChSeTo20}. The aim is also not to focus on vector-valued harmonic analysis as such; there are already various excellent sources concerning classical and recent developments in this area (for example, \cite{Kunstmann-Weis04,HyNeVeWe16,HyNeVeWe17}). Instead, we choose to focus on a few recently established \emph{connections} between the two areas. Whereas it has long been known that there are direct links between stability theory and vector-valued harmonic analysis, those connections are of a different nature than the ones presented here. Moreover, to the author's best knowledge, there are few non-technical texts that focus specifically on connections between the two areas. 

Although the ideas presented here can in principle be distilled from the original articles, we will attempt to draw a line from classical work to more recent developments, in a way which will hopefully add to what can be found in the statements and proofs of the results themselves. A concrete example of this concerns Section \ref{sec:refined}. Most of the results there involve Hilbert spaces, on which the relevant vector-valued harmonic analysis involves only Plancherel's theorem. Nonetheless, as we will attempt to explain, even there the connection to $(L^{p},L^{q})$ Fourier multiplier theory plays a role, albeit it under the surface.

Our target audience consists, for the most part, of researchers in evolution equations, although we have attempted to make the text accessible for students and researchers from related areas. For example, although many of the theorems involve concepts from Banach space geometry, their specific definitions are not crucial by any means; it is the connection between the areas that we wish to highlight. For many applications the natural examples will be $L^{p}$ spaces, but even for these concrete spaces notions from Banach space geometry arise naturally in the proofs.

In conclusion, we hope that this article can serve as an introduction to, and overview of, some recent developments connecting stability theory and vector-valued harmonic analysis, to complement the more technical statements and proofs that can be found elsewhere.

\subsection{Acknowledgments}

This article is based on the author's habilitation essay. The author would like to thank Piotr Biler, Tomasz Kania, Christian Le Merdy and Mieczysław Mastylo for their reviews. The author is also particularly grateful to David Seifert, Reinhard Stahn and Mark Veraar for the collaborations which led to some of the work discussed here, as well as for helpful remarks, and to Charles Batty, Ralph Chill and Yuri Tomilov for inspiration. Finally, the author would like to thank the anonymous referee for useful comments.

\subsection{Notation}

The natural numbers are $\N=\{1,2,\ldots\}$, and $\Z_{+}:=\N\cup\{0\}$. The H\"{o}lder conjugate of $p\in[1,\infty]$ is denoted by $p'$, so that $\frac{1}{p}+\frac{1}{p'}=1$.

The space of bounded linear operators between Banach spaces $X$ and $Y$ is $\La(X,Y)$, and $\La(X):=\La(X,X)$. The spaces of $X$-valued Schwartz functions and tempered distributions on $\Rn$ are $\Sw(\Rn;X)$ and $\Sw'(\Rn;X)$, respectively, and for $X=\C$ we simply write $\Sw(\Rn)=\Sw(\Rn;\C)$ and $\Sw'(\Rn)=\Sw'(\Rn;\C)$. The Fourier transform of $f\in\Sw'(\Rn;X)$ is denoted by $\F f$ or $\widehat{f}$. If $f\in\Ell^{1}(\Rn;X)$ then
\begin{align*}
\F f(\xi)=\int_{\Rn}e^{-i x\cdot \xi}f(x)\ud x,\quad \xi\in\Rn.
\end{align*}
The identity operator on $X$ is denoted by $\I_{X}$, and we typically write $\lambda$ for $\lambda\I_{X}$ when $\lambda\in\C$. 
The domain of a closed operator $A$ on $X$ is $\D(A)$, a Banach space with the norm
\[
\norm{x}_{\D(A)}:=\norm{x}_{X}+\norm{Ax}_{X},\quad x\in \D(A).
\]
The spectrum of $A$ is $\sigma(A)$, and the resolvent set is $\rho(A)=\C\setminus \sigma(A)$. We write $R(\lambda,A)=(\lambda -A)^{-1}$ for the resolvent operator of $A$ at $\lambda\in\rho(A)$.

We write $f(s)\lesssim g(s)$ to indicate that $f(s)\leq Cg(s)$ for all $s$ and a constant $C>0$ independent of $s$, and similarly for $f(s)\gtrsim g(s)$ and $g(s)\eqsim f(s)$.

For the required background on evolution equations and $C_{0}$-semigroups we refer to \cite{Engel-Nagel00,ArBaHiNe11}.

\section{Discussion of the problems}\label{sec:previous}

In this section we discuss Problems \ref{it:problem1}, \ref{it:problem2} and \ref{it:problem3} in more detail. We will mostly focus on developments that preceded the main thrust of this article.

\subsection{Discussion of \ref{it:problem1}}\label{subsec:problem1}

To start off, it is illustrative to examine for a moment a classical problem from harmonic analysis. 

\subsubsection{Fourier multipliers}

Given $p\in[1,\infty)$ and a measurable function $m:\Rn\to\C$ of temperate growth, one would like to determine conditions on the \emph{symbol} $m$ such that the \emph{Fourier multiplier operator}
\begin{equation}\label{eq:Fouriermult}
T_{m}(f):=\F^{-1}(m\cdot\F(f)),\quad f\in\Sw(\Rn),
\end{equation}
extends boundedly to $L^{p}(\Rn)$. For $p=2$, it follows from Plancherel's theorem that $T_{m}\in \La(L^{2}(\Rn))$ if and only if $m\in L^{\infty}(\Rn)$, with
\[
\|T_{m}\|_{\La(L^{2}(\Rn))}=\|m\|_{L^{\infty}(\Rn)}.
\]
A simple criterion also exists for $p=1$; one has $T_{m}\in \La(L^{1}(\Rn))$ if and only if $m\in L^{1}(\Rn)$. However, this problem is highly nontrivial for $1<p<\infty$, and a concrete characterization of all bounded multipliers is not known. On the other hand, a useful sufficient condition is given by the Mikhlin multiplier theorem (see~\cite[Theorem 5.2.7]{Grafakos14a}). For example, in the case where $n=1$, if $m\in C^{1}(\R\setminus\{0\})$ and
\begin{equation}\label{eq:Mikhlinscalar}
\sup_{\xi\neq0}|m(\xi)|+|\xi m'(\xi)|<\infty,
\end{equation}
then $T_{m}\in \La(L^{p}(\R))$ for all $1<p<\infty$.

It turns out that this classical problem from harmonic analysis also has applications to evolution equations, most notably to questions of maximal $L^{p}$ regularity (see~\cite{Kunstmann-Weis04}). Here one generalizes \eqref{eq:Fouriermult} to the infinite dimensional setting by considering Banach spaces $X$ and $Y$, operator-valued symbols $m:\Rn\to\La(X,Y)$, and the corresponding operator
\[
T_{m}(f):=\F^{-1}(m\cdot\F(f)),\quad f\in\Sw(\Rn;X).
\]
For $p,q\in[1,\infty]$ we write $m\in\Ma_{p,q}(X,Y)$ if there exists a $C\geq0$ such that $\|T_{m}(f)\|_{L^{q}(\Rn;Y)}\leq C\|f\|_{L^{p}(\Rn;X)}$ for all $f\in\Sw(\Rn;X)$, and we say that $m$ is a \emph{bounded $(L^{p}(\Rn;X),L^{q}(\Rn;Y))$ Fourier multiplier}. For $p<\infty$ this implies that $T_{m}$ extends to a bounded map from $L^{p}(\Rn;X)$ to $L^{q}(\Rn;Y)$. For $p=\infty$ the Schwartz functions are not dense in $L^{\infty}(\Rn;X)$, but the class of $(L^{\infty}(\Rn;X),L^{q}(\Rn;Y))$ Fourier multipliers is sufficiently small that the corresponding operator does in fact extend to $L^{\infty}(\Rn;X)$, in a different manner. Moreover, the case $p=\infty$ will not play an explicit role for us, and using this terminology simplifies the statements of the results. Accordingly, we write 
\begin{equation}\label{eq:multipliernorm}
\|m\|_{\Ma_{p,q}(X,Y)}:=\|T_{m}\|_{\La(L^{p}(\Rn;X),L^{q}(\Rn;Y))}
\end{equation}
for $m\in\Ma_{p,q}(X,Y)$, and $\Ma_{p,q}(X):=\Ma_{p,q}(X,X)$. For simplicity, we omit the dimension $n$ in this notation.

A generalization of the Mikhlin condition \eqref{eq:Mikhlinscalar} to the operator-valued setting was obtained by Weis in~\cite{Weis01}. More precisely, for $n=1$, one has $m\in\Ma_{p,p}(X,Y)$ for all $1<p<\infty$ if $X$ and $Y$ are UMD spaces, $m\in C^{1}(\R\setminus\{0\};\La(X,Y))$, and if
\begin{equation}\label{eq:Weis}
\{m(\xi)\mid \xi\in\R\setminus\{0\}\} \text{ and }\{\xi m'(\xi)\mid \xi\in\R\setminus\{0\}\}
\end{equation}
are $R$-bounded in $\La(X,Y)$. For the definition of $R$-boundedness and the UMD property we refer to~\cite{HyNeVeWe16,HyNeVeWe17}. A specific example of an operator-valued symbol to which this theory applies is
\[
m(\xi):=R(i\xi,A),\quad\xi\in\R,
\]
where $A$ generates a $C_{0}$-semigroup and $\{\lambda\in\C\mid \Real(\lambda)\geq0\}\subseteq\rho(A)$. The multiplier condition of Weis, for this specific symbol, is the main tool in the theory of maximal $L^{p}$ regularity \cite{Kunstmann-Weis04}.

\subsubsection{Stability theory}

The resolvent also appears as a Fourier multiplier in the asymptotic theory of evolution equations. Indeed, suppose as before that $A$ generates a uniformly bounded $C_{0}$-semigroup $(T(t))_{t\geq0}$, and assume additionally that \eqref{eq:supremumbound} holds: $\sup_{\xi\in\R}\|R(i\xi,A)\|_{\La(X)}<\infty$. Then we may fix an $\w<0$ such that \eqref{eq:supremumbound2} holds: $\sup_{\xi\in\R}\|R(\w+i\xi,A)\|_{\La(X)}<\infty$. Now let $\wt{\w}>0$, and this time set 
\begin{equation}\label{eq:mstability}
m(\xi):=I_{X}+(\wt{\w}-\w) R(\w+i\xi,A),\quad\xi\in\R.
\end{equation}
Then
\begin{equation}\label{eq:deff}
t\mapsto f(t):=\begin{cases}
e^{-\wt{w} t}T(t)x&t\geq0,\\
0&t<0,
\end{cases}
\end{equation}
is an element of $L^{p}(\R;X)$ for all $x\in X$ and $1\leq p\leq \infty$. Moreover, \eqref{eq:Laplace} and the resolvent identity yield
\[
m(\xi)\F f(\xi)=(I_{X}+(\wt{\w}-\w)R(\w+i\xi,A))R(\wt{\w}+i\xi,A)x=R(\w+i\xi,A)x.
\]
Hence, for $x\in D((-A)^{\alpha})$ for $\alpha>1$, one can use \eqref{eq:integrability} to take the inverse Fourier transform in \eqref{eq:inverseFourier} and obtain
\begin{equation}\label{eq:imagemult}
e^{-\w t}T(t)x=\F^{-1}(m\cdot\F f)(t)=T_{m}(f)(t)
\end{equation}
for $t\geq0$. Since the resolvent is absolutely integrable for $x\in D((-A)^{\alpha})$ with $\alpha>1$, this in turn yields \eqref{eq:asymptotic} for such $x$: 
\begin{equation}\label{eq:decayBanach}
\|T(t)x\|_{X}\leq C'e^{\w t}\|x\|_{D((-A)^{\alpha})},\quad t\geq0.
\end{equation}
On the other hand, if $m$ is a bounded $(L^{p}(\Rn;X),L^{p}(\Rn;X))$ Fourier multiplier for some $p\in[1,\infty]$, then it is not too difficult to combine \eqref{eq:imagemult} with the density of $D((-A)^{\alpha})$ in $X$ to extend \eqref{eq:asymptotic} to all $x\in X$. This abstract connection between Fourier multiplier theory and the asymptotic theory of evolution equations can be found in \cite{Latushkin-Shvydkoy01}, for example, although an equivalent formulation in terms of convolutions goes back further (see e.g.~\cite[Section 3.3]{vanNeerven96b}).

Now, if $X$ is a Hilbert space, then Plancherel's theorem and \eqref{eq:supremumbound2} imply that $m\in\Ma_{2,2}(X,X)$, thereby recovering the Gearhart--Pr\"{u}ss theorem. However, this argument no longer works if $X$ is not a Hilbert space. On the other hand, many spaces which arise in applications are themselves $L^{q}$ spaces over some measure space, and for $1<q<\infty$ these are UMD spaces. Hence it is natural to attempt to apply Weis' theorem. Unfortunately, although assumption \eqref{eq:Weis} is natural for the theory of maximal regularity, it is not particularly useful for applications to the asymptotic theory of semigroups. Indeed, the assumption that 
\[
\{\xi m'(\xi)\mid \xi\in\R\setminus\{0\}\}=\{-i(\wt{\w}-\w)\xi R(\w+i\xi,A)^{2}\mid \xi\in\R\setminus\{0\}\}\subseteq\La(X)
\]
 is $R$-bounded, or even just uniformly bounded, comes very close to assuming that $A$ generates an analytic semigroup. Moreover, the (exponential) asymptotic theory of analytic semigroups is well known and much simpler than the general case, cf.~\cite[Theorem 5.1.12]{ArBaHiNe11}. 
 
For a long time, apart from some exceptions discussed below, this obstacle meant that operator-valued Fourier multipliers were of little practical use for the asymptotic theory of evolution equations.

\subsection{Discussion of \ref{it:problem2}}\label{subsec:problem2}

The study of more refined asymptotic behavior arises naturally from applications to damped wave equations, as will be explained here. 

\subsubsection{Motivation for the problem}

One says that a subset $U$ of a compact connected Riemannian manifold $M$ without boundary satisfies the \emph{geometric control condition} if there exists an $L>0$ such that every geodesic of length at least $L$ intersects $U$. It was first shown by Rauch and Taylor in~\cite{Rauch-Taylor75} that the energy of solutions to a damped wave equation, with smooth nonzero damping function $b:M\to[0,\infty)$, decays exponentially if there exists a $c>0$ such that the set $\{x\in M\mid b(x)\geq c\}$ satisfies the geometric control condition. It follows from~\cite{Ralston69} that the geometric control condition is also necessary for exponential energy decay, and similar statements hold on manifolds with boundary (see~\cite{BaLeRa92,Burq-Gerard97}). Moreover, it is easy to construct examples where the damping region does not satisfy the geometric control condition; simply ensure that there is no damping on an open neighborhood of at least one geodesic.

We will not go into detail regarding the connection between the geometric control condition and exponential energy decay, but we wish to point out that it is an instance of the phenomenon of \emph{propagation of singularities}. When applied to the Laplacian on a manifold, this phenomenon dictates that singularities of initial data propagate along geodesics under the wave equation. An equivalent, and often more useful, statement is that regularity of initial data propagates in the same manner. On a quantitative level, modulo a lower-order Sobolev norm, one can bound the space-time $L^{2}$ norm of the solution on a neighborhood of a geodesic in terms of the $L^{2}$ norm on a neighborhood of a single point on the geodesic. By combining this with the geometric control condition for the damping region, one can derive exponential energy decay. The converse implication follows by constructing approximate solutions to the wave equation which concentrate along a single geodesic.

Returning to the main topic of this article, when formulating a damped wave equation as an evolution equation on $X=L^{2}(M)\times W^{1,2}(M)$, the solution semigroup $(T(t))_{t\geq0}$ is uniformly bounded, its generator $A$ satisfies $i\R\subseteq\rho(A)$, and for each $x\in X$ one has $T(t)x\to0$ as $t\to\infty$. However, the uniform boundedness condition \eqref{eq:supremumbound} for the resolvent holds if and only the damping region satisfies the geometric control condition. Hence, when this condition is not satisfied, the arguments from the previous section do not apply, and one expects to encounter more refined decay behavior (see e.g.~\cite{Lebeau96,Anantharaman-Leautaud14,Burq98}). 

Heuristically, it follows from \eqref{eq:inverseFourier}, with $\w=0$, that growth rates for the resolvent are inversely proportional to decay rates for the semigroup. However, when formalizing this heuristic, there are two obstacles. Firstly, in general one cannot expect a uniform decay rate to hold for all $x\in X$, since a simple argument shows that such a uniform rate would imply exponential decay and \eqref{eq:supremumbound}. Hence it is natural to consider initial data $x\in D(A)$, for which the semigroup orbit $t\mapsto T(t)x$ is a classical solution to the associated abstract Cauchy problem \eqref{eq:evoleq}. Secondly, the generator of the solution semigroup to a damped wave equation is not a normal operator, and therefore it is nontrivial to determine the precise relationship between decay rates for the semigroup and growth rates for the resolvent. 

Note that it is highly relevant in practice to determine the relationship between resolvent growth rates and semigroup decay rates. Indeed, as indicated before, the resolvent is more tractable than the semigroup, and in practical applications one typically proceeds by obtaining suitable bounds for the resolvent. These yield decay rates for classical solutions to the evolution equation, and one would like to know exactly what decay rate corresponds to the observed resolvent growth rate.

\subsubsection{Refined decay rates}

Next, let us make this discussion more precise. Let $A$ generate a uniformly bounded $C_{0}$-semigroup $(T(t))_{t\geq0}$ on a Banach space $X$ such that $i\R\subseteq\rho(A)$, and let $M:[0,\infty)\to[0,\infty)$ be such that $M(\lambda)\to\infty$ as $\lambda\to\infty$, and 
\begin{equation}\label{eq:defM}
\|R(i\xi,A)\|_{\La(X)}\leq M(|\xi|),\quad\xi\in\R.
\end{equation}
To determine uniform decay rates for classical solutions to \eqref{eq:evoleq}, one has to obtain norm estimates for $\|T(t)A^{-1}\|_{\La(X)}$, since 
\[
\|T(t)x\|_{X}\leq \|T(t)A^{-1}\|_{\La(X)}\|Ax\|_{X}\leq \|T(t)A^{-1}\|_{\La(X)}\|x\|_{D(A)}
\]
for all $x\in D(A)$.

In~\cite{BaEnPrSc06}, Batkai, Engel, Pr\"{u}ss and Schnaubelt obtained some early results on the relationship between growth rates for $M$ and decay rates for the semigroup. They considered $M$ such that $M(\lambda)\eqsim \lambda^{\alpha}$ as $\lambda\to\infty$, for some $\alpha>0$, and showed that in this case \eqref{eq:defM} implies that for all $\veps>0$ there exists a $C_{\veps}\geq 0$ such that
\[
\|T(t)A^{-1}\|_{X}\leq C_{\veps} \frac{1}{t^{\frac{1}{\alpha}-\veps}},\quad t\to\infty.
\]
They also proved a converse implication, again with a polynomial $\veps$ loss. On Hilbert spaces, sharper estimates were obtained by Liu and Rao in~\cite{Liu-Rao05}.

This early work was improved upon by Batty and Duyckaerts in~\cite{Batty-Duyckaerts08}, using an approach from Tauberian theory. Suppose without loss of generality that
\[
M(\lambda)=\sup_{|\xi|\leq\lambda}\|R(i\xi,A)\|_{\La(X)},\quad\lambda\geq0,
\]
and set
\[
M_{\log}(\lambda):=M(\lambda)\big(\log(1+M(\lambda))+\log(1+\lambda)\big),\quad \lambda\geq0.
\]
Since $M$ and $M_{\log}$ are non-decreasing, they have (possibly non-unique) right inverses $M^{-1}$ and $M_{\log}^{-1}$, respectively. Batty and Duyckaerts showed that there exist $c,C>0$ such that
\begin{equation}\label{eq:BattyDuyck}
\frac{c}{M^{-1}(Ct)}\leq \|T(t)A^{-1}\|_{\La(X)}\leq \frac{C}{M_{\log}^{-1}(ct)}
\end{equation}
for sufficiently large $t$. For polynomial $M$, the left-hand side differs from the right-hand side by a logarithmic factor. They also conjectured that the second inequality in \eqref{eq:BattyDuyck} is sharp on general Banach spaces, but that it can be improved on Hilbert spaces.

Their conjecture was, for the most part, proved to be correct by Borichev and Tomilov in the seminal paper~\cite{Borichev-Tomilov10}. They showed that, indeed, the right-hand side of \eqref{eq:BattyDuyck} cannot be improved on general Banach spaces. Moreover, if $X$ is a Hilbert space and if $M(\lambda)\eqsim \lambda^{\alpha}$ for $\lambda\geq1$, then 
\begin{equation}\label{eq:noloss}
\|T(t)A^{-1}\|_{\La(X)}\leq \frac{C}{M^{-1}(ct)}
\end{equation}
for certain $c,C>0$ and for sufficiently large $t$. Combined with the first inequality of \eqref{eq:BattyDuyck} this shows, up to constants, exactly what the best decay rate is which can be expected from a classical solution: one has $\|T(t)A^{-1}\|_{\La(X)}\eqsim t^{-1/\alpha}$ as $t\to\infty$. This work has been used extensively to obtain rates of decay for damped wave equations (see e.g.~the references in~\cite{BaChTo16,ChSeTo20}).

Note that, although the case of polynomially growing $M$ is of the most interest for applications,~\cite{Borichev-Tomilov10} did not fully resolve the conjecture by Batty and Duyckaerts. Moreover, one can construct examples showing that \eqref{eq:noloss} does not hold on Hilbert spaces for all $M$, even if $(T(t))_{t\geq0}$ is assumed to be a normal semigroup. It thus remained an open question to determine for which $M$ \eqref{eq:noloss} holds. In~\cite{BaChTo16}, Batty, Chill and Tomilov used advanced tools from functional calculus theory to obtain \eqref{eq:noloss} for certain so-called \emph{regularly varying} $M$, a class which strictly contains the polynomials. They also showed that \eqref{eq:noloss} holds for all normal semigroups $(T(t))_{t\geq0}$ if and only if $M$ has a property called \emph{positive increase}. These properties are introduced in Definition~\ref{def:posinc}, but here we note that every regularly varying function (of positive index) has positive increase, whereas not every function of positive increase is regularly varying. Hence \cite{BaChTo16} did not completely close the gap between the necessary and sufficient conditions for \eqref{eq:noloss}.

\subsection{Discussion of \ref{it:problem3}}\label{subsec:problem3}

So far, we have discussed two problems concerning decay rates for semigroup orbits, under the assumption that the imaginary axis is contained in the resolvent set of the generator. By contrast, Problem \ref{it:problem3} concerns asymptotic behavior in a setting where this assumption is dropped. Instead, we assume that the resolvent is bounded on imaginary lines, and that it blows up at a specified rate as those imaginary lines approach the imaginary axis. In this case, one can only hope to derive more refined \emph{growth} behavior for semigroup orbits. 

\subsubsection{The Kreiss condition}\label{subsubsec:Kreiss}

To motivate this problem, we make the setting more precise. Let $(T(t))_{t\geq0}$ be a $C_{0}$-semigroup with generator $A$ on a Banach space $X$, and suppose that $\lambda\in\rho(A)$ for all $\lambda\in\C$ with $\Real(\lambda)>0$. In applications, one often wants to know whether $(T(t))_{t\geq0}$ is uniformly bounded. The Hille--Yosida theorem provides a necessary and sufficient condition for this to be true: there exists a $C\geq0$ such that 
\begin{equation}\label{eq:HilleYosida}
\|R(\lambda,A)^{n}\|_{\La(X)}\leq \frac{C}{(\Real(\lambda))^{n}}
\end{equation}
for all $n\in\N$ and $\lambda\in\C$ with $\Real(\lambda)>0$. Although the resolvent is much more tractable than the semigroup itself, \eqref{eq:HilleYosida} requires an estimate for all powers of the resolvent, which makes it inconvenient for many applications. On the other hand, $(T(t))_{t\geq0}$ is contractive if and only if 
\begin{equation}\label{eq:contraction}
\|R(\lambda,A)\|_{\La(X)}\leq \frac{1}{\Real(\lambda)}
\end{equation}
for all $\lambda\in\C$ with $\Real(\lambda)>0$. Clearly this is a much more convenient condition to check than \eqref{eq:HilleYosida}. 

It is natural to ask what can be said about $(T(t))_{t\geq0}$ given the following intermediate condition between \eqref{eq:HilleYosida} and \eqref{eq:contraction}: there exists a $C\geq0$ such that
\begin{equation}\label{eq:Kreiss}
\|R(\lambda,A)\|_{\La(X)}\leq \frac{C}{\Real(\lambda)}
\end{equation}
for all $\lambda\in\C$ with $\Real(\lambda)>0$. This condition was introduced in~\cite{Kreis59}, and it is nowadays called the \emph{Kreiss condition}. If $X$ is a Hilbert space then it guarantees that $\|T(t)\|_{\La(X)}$ grows at most linearly as $t\to\infty$, while for a general Banach space $X$ it is possible that $\|T(t)\|_{\La(X)}$ grows exponentially as $t\to\infty$ (see~\cite{EisZwa06}). 

In various applications one encounters semigroups $(T(t))_{t\geq0}$ such that $\|T(t)\|_{\La(X)}$ grows polynomially as $t\to\infty$, and the discussion above then begs the question how such refined growth behavior is related to estimates for the resolvent. More precisely, let $g:(0,\infty)\to(0,\infty)$ be non-decreasing and such that
\begin{equation}\label{eq:Kreissgeneral}
\|R(\lambda,A)\|_{\La(X)}\leq g\Big(\frac{1}{\Real(\lambda)}\Big)
\end{equation}
for all $\lambda\in\C$ with $\Real(\lambda)>0$. One expects that the rate of resolvent blow-up in \eqref{eq:Kreissgeneral} is related to the growth of the associated semigroup, as is the case for the Kreiss condition~\eqref{eq:Kreiss}. In~\cite{Helffer-Sjostrand10} (see also \cite{Helffer-Sjostrand21}), Helffer and Sj\"{o}strand showed that this is indeed the case on Hilbert spaces, but it is not immediately clear how to extend their approach to more general Banach spaces.

In fact, to conclude this section and to connect Problems \ref{it:problem1}, \ref{it:problem2} and \ref{it:problem3}, we remark that it is not at all clear from the techniques discussed so far whether there is a connection between the theory of operator-valued Fourier multipliers, which arises naturally in \ref{it:problem1} when considering exponential decay of semigroups, and subtle asymptotic behavior as in \ref{it:problem2} and \ref{it:problem3}.

\section{Operator-valued $(L^{p},L^{q})$ multipliers}\label{sec:multipliers}

In this section we discuss Development \ref{it:develop1}. This work arose from the realization that several known results concerning Problem \ref{it:problem1}, which had been obtained using disparate methods, could be proved in a unified manner by framing them in terms of $(L^{p},L^{q})$ Fourier multiplier properties of the resolvent. Doing so also yielded new results on exponential stability of semigroups on non-Hilbertian Banach spaces. However, not much was known about the theory of operator-valued $(L^{p},L^{q})$ Fourier multipliers, perhaps for reasons which will be explained below. In particular, it seems that concrete multiplier conditions and a connection to Banach space geometry were not known prior to \cite{Rozendaal-Veraar18,Rozendaal-Veraar17b}. We will focus on the theory developed in those articles, but we refer to \cite{Dominguez-Veraar21} for a recent contribution in a similar direction.

We will first discuss operator-valued Fourier multipliers between $L^{p}(\Rn;X)$ and $L^{q}(\Rn;Y)$, for Banach spaces $X$ and $Y$, and then focus on multipliers between appropriate vector-valued Besov spaces. Multipliers of the latter type arise naturally when considering certain endpoint cases in the $(L^{p},L^{q})$ theory.

\subsection{Scalar multipliers}\label{subsec:scalarmultipliers1}

In the scalar case, the theory of $(L^{p}(\Rn),L^{q}(\Rn))$ Fourier multipliers for $p\neq q$ is relatively straightforward. If $p>q$, then the only bounded Fourier multiplier from $L^{p}(\Rn)$ to $L^{q}(\Rn)$ is the zero operator. For $p<q$, one can always combine $(L^{p}(\Rn),L^{p}(\Rn))$ multiplier conditions with Sobolev embeddings to obtain $(L^{p}(\Rn),L^{q}(\Rn))$ multipliers. However, loosely speaking, to obtain boundedness results for a large class of Fourier multipliers, it is necessary that $p\leq 2\leq q$ (see~\cite{Hormander60}). In this case, let $m:\Rn\to\C$ be measurable and such that
\begin{equation}\label{eq:condm}
\sup_{\xi\neq0}|\xi|^{n(\frac{1}{p}-\frac{1}{q})}|m(\xi)|<\infty.
\end{equation}
For $s\in\R$, write $|D|^{s}$ for the Fourier multiplier with symbol $\xi\mapsto |\xi|^{s}$. Let $s_{1},s_{2}\in[1,\infty]$ be such that $\frac{1}{s_{1}}=n(\frac{1}{p}-\frac{1}{2})$ and $\frac{1}{s_{2}}=n(\frac{1}{2}-\frac{1}{q})$, and let $\dot{H}^{p}_{s_{1}}(\Rn)=|D|^{-s_{1}}L^{p}(\Rn)$ and $\dot{H}^{2}_{s_{2}}(\Rn)=|D|^{-s_{2}}L^{q}(\Rn)$ be homogeneous Bessel potential spaces\footnote{We use this terminology, as opposed to the more common ``homogeneous Sobolev spaces", because in the vector-valued setting Bessel potential spaces do not in general coincide with Sobolev spaces, as defined in the classical sense.}. Then one can use Sobolev embeddings, Plancherel's theorem, and \eqref{eq:condm} to write
\begin{equation}\label{eq:classicalmult}
\begin{aligned}
\|T_{m}(f)\|_{L^{q}(\Rn)}&\lesssim \|T_{m}(f)\|_{\dot{H}^{2}_{s_{2}}(\Rn)}=\big\||D|^{s_{2}}T_{m}(f)\big\|_{L^{2}(\Rn)}\\
&\lesssim \big\||D|^{-s_{1}}f\big\|_{L^{2}(\Rn)}=\big\|f\big\|_{\dot{H}^{2}_{-s_{1}}(\Rn)}\lesssim\|f\|_{L^{p}(\Rn)}
\end{aligned}
\end{equation}
for all $f\in\Sw(\Rn)$, thereby showing that $T_{m}:L^{p}(\Rn)\to L^{q}(\Rn)$ is bounded. Note that, unlike in the case where $p=q\neq 2$, one does not require any smoothness of $m$ for the resulting multiplier theorem. Although it is possible to obtain more sophisticated results, this procedure already yields a substantial class of useful $(L^{p}(\Rn),L^{q}(\Rn))$ Fourier multipliers.

The same procedure works for $(L^{p}(\Rn;X),L^{q}(\Rn;Y))$ multipliers with operator-valued symbols $m:\Rn\to\La(X,Y)$, if $X$ and $Y$ are Hilbert spaces. However, in such a case it is typically more useful to apply Plancherel's theorem for $p=q=2$. On the other hand, if $X$ or $Y$ is not a Hilbert space, then this approach fundamentally breaks down. 

It is possible that little was known about operator-valued $(L^{p}(\Rn;X),L^{q}(\Rn;Y))$ Fourier multipliers for three reasons. Firstly, the theory is relatively straightforward for $X=Y=\C$. Secondly, it is not immediately clear what applications such a theory has. And thirdly, it is not clear how to address the fact that the techniques from the scalar case seemingly break down in a fundamental manner.

Nonetheless, it turns out to be possible to extend the scalar $(L^{p},L^{q})$ Fourier multiplier theory to the operator-valued setting, under appropriate geometric assumptions on the underlying Banach spaces. Interestingly, the UMD assumption on the underlying spaces, which is so crucial for the vector-valued theory when $p=q$, hardly plays a role when $p\neq q$. Instead, the geometry of the underlying spaces $X$ and $Y$ comes in through properties such as Fourier type, and type and cotype. These properties are, to some extent, more ``quantitative" than the UMD property, in the sense that they involve parameters $p_{X}\in[1,2]$ and $q_{Y}\in[2,\infty]$ that are connected to those $p$ and $q$ for which one can obtain $(L^{p}(\Rn;X),L^{q}(\Rn;Y))$ multiplier theorems.

\subsection{Operator-valued multipliers}\label{subsec:operatorvalued}

Various techniques can be used to extend the scalar theory to the operator-valued setting. Here we highlight two such extensions. The first one, Proposition 3.9 in \cite{Rozendaal-Veraar18}, is almost trivial to prove, but it is nonetheless quite useful for applications. One says that a Banach space $X$ has \emph{Fourier type} $p\in[1,2]$ if $\F:L^{p}(\Rn;X)\to L^{p'}(\Rn;X)$ boundedly (for more on this  notion see \cite{HyNeVeWe16}). Recall that we write $m\in\Ma_{p,q}(X,Y)$ if $T_{m}:L^{p}(\Rn;X)\to L^{q}(\Rn;Y)$ is bounded.

\begin{proposition}\label{prop:Fouriertype}
Let $X$ be a Banach space with Fourier type $p\in[1,2]$, and $Y$ a Banach space with Fourier type $q'$ for $q\in[2,\infty]$. Let $r\in[1,\infty]$ be such that $\frac{1}{r}=\frac{1}{p}-\frac{1}{q}$, and let $m:\Rn\to\La(X,Y)$ be strongly measurable and such that $[\xi\mapsto\|m(\xi)\|_{\La(X,Y)}]\in L^{r}(\Rn)$. Then $m\in\Ma_{p,q}(X,Y)$.
\end{proposition}

The proof takes one line: since $\frac{1}{r}=\frac{1}{q'}-\frac{1}{p'}$, one can use H\"{o}lder's inequality and the assumptions on $X$ and $Y$ to write
\begin{align*}
\|T_{m}(f)\|_{L^{q}(\Rn;Y)}\lesssim \|m\F(f)\|_{L^{q'}(\Rn;Y)}\lesssim \|\F(f)\|_{L^{p'}(\Rn;X)}\lesssim \|f\|_{L^{p}(\Rn;X)}
\end{align*}
for each $f\in\Sw(\Rn;X)$.

Note that the assumption on $m$ is satisfied if there exists an $\veps>0$ such that
\begin{equation}\label{eq:condm3}
\sup_{\xi\neq 0}\max(|\xi|,|\xi|^{-1})^{n(\frac{1}{p}-\frac{1}{q})+\veps}\|m(\xi)\|_{\La(X,Y)}<\infty,
\end{equation}
thereby providing an extension of \eqref{eq:condm} up to an $\veps$ loss.

Next, we discuss a result which concerns another analogue of condition \eqref{eq:condm}. Its proof is more involved and more closely related to the strategy in the scalar case, except that it relies on the notions of type and cotype of a Banach space (see~\cite{HyNeVeWe17}). Although one cannot directly rely on Plancherel's theorem unless $X$ and $Y$ are Hilbert spaces, it was realized by Kalton and Weis that there are ways to introduce a Hilbert space structure on general Banach spaces. Namely, the theory of \emph{$\gamma$-radonifying operators} (see~\cite{Kalton-Weis04,vanNeerven10}) involves function spaces $\gamma(\Rn;X)$ and $\gamma(\Rn;Y)$ on which versions of Plancherel's theorem are available, as well as other properties from the Hilbert space setting. Moreover, under type and cotype assumptions on $X$ and $Y$, these $\gamma$-spaces embed into the Bessel potential spaces $\dot{H}^{s}_{p}(\Rn;X)$ and $\dot{H}^{t}_{q}(\Rn;X)$, cf.~\cite{Veraar13}. Then, using a similar argument as in \eqref{eq:classicalmult}, one can prove the following result, Theorem 3.18 in \cite{Rozendaal-Veraar18}. 

\begin{theorem}\label{thm:typecotype}
Let $X$ and $Y$ be Banach spaces such that $X$ has type $p_{0}\in(1,2]$, and $Y$ has cotype $q_{0}\in[2,\infty)$. Let $p\in(1,p_{0})$ and $q\in(q_{0},\infty)$, and let $m:\Rn\setminus\{0\}\to\La(X,Y)$ be strongly measurable and such that
\begin{equation}\label{eq:condm2}
\{|\xi|^{n(\frac{1}{p}-\frac{1}{q})}m(\xi)\mid \xi\in\Rn\setminus\{0\}\}\subseteq\La(X,Y)
\end{equation}
is $R$-bounded. Then $m\in \Ma_{p,q}(X,Y)$. If $p_{0}=2$ or $q_{0}=2$, then one can let $p=p_{0}$ or $q=q_{0}$, respectively.
\end{theorem}

The type and cotype assumptions in Theorem \ref{thm:typecotype} are, up to the limiting exponents, necessary (see \cite[Section 3.6]{Rozendaal-Veraar18}). The $R$-boundedness assumption on \eqref{eq:condm2} is used to ensure that multiplication by $m$ maps $\gamma(\Rn;X)$ to $\gamma(\Rn;Y)$. Such an assumption is also necessary in general. However, if $X=Y$ and $m$ is a scalar-valued symbol which acts on $X$ by multiplication, then the $R$-boundedness condition on \eqref{eq:condm2} reduces to assumption \eqref{eq:condm}, by the Kahane contraction principle. It appears that, even for scalar-valued multipliers, Theorem \ref{thm:typecotype} was new in the vector-valued setting.

It is not clear whether one can let $p=p_{0}$ or $q=q_{0}$ in Theorem \ref{thm:typecotype} for $p_{0}<2$ or $q_{0}>2$, respectively. Hence this result appears to yield a similar $\veps$ loss with respect to \eqref{eq:condm} as Proposition \ref{prop:Fouriertype}, and with an additional $R$-boundedness assumption. However, although the notions of Fourier type, type and cotype are closely related, they are not the same. Firstly, a Banach space $X$ with Fourier type $p$ has type $p$ and cotype $p'$. However, $X$ may in fact have a larger type and smaller cotype than is guaranteed by its Fourier type. For example, if $X=L^{\sigma}(\Omega)$ for a measure space $\Omega$ and $\sigma\in[1,\infty)$, then $X$ has Fourier type $p=\min(\sigma,\sigma')$, type $p_{0}=\min(\sigma,2)$, and cotype $q_{0}=\max(\sigma',2)$. This means that for Lebesgue spaces, which arise most commonly in applications, the decay assumption on $m$ in Theorem \ref{thm:typecotype} is substantially weaker than that in Proposition \ref{prop:Fouriertype}.

\subsection{Multipliers on Besov spaces}\label{subsec:Besov}

It is possible to avoid the $\veps$ loss in \eqref{eq:condm3}, and reach the limiting exponents in Theorem \ref{thm:typecotype}, by working with vector-valued Besov spaces (for more on these see e.g.~\cite{Amann97}). Operator-valued Fourier multipliers on Besov spaces have been considered in~\cite{Amann97,Girardi-Weis03,Hytonen04}, but there the symbols satisfy versions of Mikhlin's theorem, and the operators map between $B^{s}_{p,v}(\Rn;X)$ and $B^{s}_{p,v}(\Rn;Y)$, for $s\in\R$ and $p,v\in[1,\infty]$. Such a setting can be viewed as the analogue for Besov spaces of $(L^{p},L^{p})$ Fourier multipliers, although on Besov spaces the UMD condition does not play a role, and in the aforementioned works it is replaced by assumptions on the Fourier type of the underlying spaces. 

Analogues for Besov spaces of the results on $(L^{p},L^{q})$ Fourier multipliers can be found in \cite{Rozendaal-Veraar17b}. Again, these rely on the geometric notions of Fourier type, type and cotype. For example, a special case of \cite[Theorem 4.3]{Rozendaal-Veraar17b} implies the following result.

\begin{theorem}\label{thm:Besov}
Let $X$ and $Y$ be Banach spaces such that $X$ has type $p\in[1,2]$, and $Y$ has cotype $q\in[2,\infty]$. Let $m:\Rn\to\La(X,Y)$ be strongly measurable and such that
\begin{equation}\label{eq:condm4}
\{(1+|\xi|)^{n(\frac{1}{p}-\frac{1}{q})}m(\xi)\mid \xi\in\Rn\}\subseteq\La(X,Y)
\end{equation}
is $R$-bounded. Then $T_{m}:B^{s}_{p,v}(\Rn;X)\to B^{s}_{q,v}(\Rn;Y)$ is bounded for all $s\in\R$ and $v\in[1,\infty]$.
\end{theorem}

An analogue of Proposition \ref{prop:Fouriertype} also holds (see \cite[Proposition 3.7]{Rozendaal-Veraar17b}).

Note that Theorem \ref{thm:Besov} indeed involves the limiting cases of the geometric assumptions in Theorem \ref{thm:typecotype}. Moreover, apart from the additional $R$-boundedness assumption near zero, \eqref{eq:condm4} is just \eqref{eq:condm2}. One can in turn get rid of this additional condition at zero by working with homogeneous Besov spaces, cf.~\cite[Theorem 4.5]{Rozendaal-Veraar18}. 

We also note that, using embeddings between Besov spaces and Bessel potential spaces, the results on Besov spaces yield an $(L^{p},L^{q})$ multiplier theorem which deals with the limiting exponents $p=p_{0}$ and $q=q_{0}$ in Theorem \ref{thm:typecotype}, but which involves a stronger condition than \eqref{eq:condm2} (see \cite[Theorem 4.6]{Rozendaal-Veraar17b}). 

On the other hand, using the embedding $B^{0}_{q,1}(\R;X)\subseteq L^{q}(\R;X)$, under the conditions of Theorem \ref{thm:Besov} one immediately obtains that $T_{m}:B^{0}_{p,1}(\Rn;X)\to L^{q}(\Rn;Y)$ is bounded. This is in turn equivalent to the statement that
\begin{equation}\label{eq:Besov}
T_{m}:B^{n(\frac{1}{p}-\frac{1}{q})}_{p,1}(\R;X)\to L^{q}(\R;Y)
\end{equation}
is bounded if $\{m(\xi)\mid \xi\in\Rn\}\subseteq\La(X,Y)$ is $R$-bounded, a fact which will be used in Appendix \ref{sec:calctypecotype}.

 The proof of Theorem \ref{thm:Besov} again uses the Hilbert space structure which is captured by the spaces $\gamma(\Rn;X)$ and $\gamma(\Rn;Y)$. However, the $\veps$-loss in the embeddings of these spaces into the scale of Bessel potential spaces, which in turn gives rise to the loss in Theorem \ref{thm:typecotype}, does not occur for functions with compact Fourier support. Since the Besov space norm involves summation over $L^{p}$-norms of functions with Fourier support in dyadic annuli, this allows for sharper embedding results of the $\gamma$-spaces into the scale of Besov spaces~\cite{KaNeVeWe08}. And this in turn allows one to reach the limiting geometric assumptions in Theorem~\ref{thm:Besov}. 

Apart from allowing one to consider critical exponents which are not available for Bessel potential spaces, there are other inherent advantages to working with Besov spaces. For example, loosely speaking, weighted orbits of $C_{0}$-groups are elements of Besov spaces if the initial data comes from an appropriate interpolation space. In~\cite{Haase-Rozendaal16,Rozendaal19}, this fact was combined with Fourier multiplier theorems on Besov spaces and applied to functional calculus theory. This development will be described in more detail in Appendix \ref{sec:calctypecotype}, although we mention here that the Fourier multiplier theorems on Besov spaces which have been discussed so far are those which are used in~\cite{Rozendaal19}.

\section{Exponential decay}\label{sec:exponential}

In this section we apply the operator-valued $(L^{p},L^{q})$ Fourier multiplier theorems from the previous section to Problem \ref{it:problem1}. 

\subsection{Exponential decay and Fourier multipliers}\label{subsec:expmult}

Our first main tool is the following connection between exponential stability and $(L^{p},L^{q})$ Fourier multipliers. It is a special case of a more general result (see \cite{Latushkin-Shvydkoy01,Hieber01,Rozendaal-Veraar18a}).

\begin{proposition}\label{prop:exponential}
Let $(T(t))_{t\geq0}$ be a $C_{0}$-semigroup with generator $A$ on a Banach space $X$. Let $\w'\in\R$ be such that $\lambda\in\rho(A)$ for all $\lambda\in\C$ with $\Real(\lambda)\geq\w'$, and such that 
\begin{equation}\label{eq:supremumbound3}
\sup_{\Real(\lambda)\geq\w'}\|R(\lambda,A)\|_{\La(X)}<\infty.
\end{equation}
Let $\alpha\geq0$. Then the following statements are equivalent:
\begin{enumerate}
\item\label{it:exp1} For each $\w>\w'$ there exist $p,q\in[1,\infty]$ such that the Fourier multiplier operator with symbol $\xi\mapsto R(\w+i\xi,A)$ is bounded from $L^{p}(\R;D((-A)^{\alpha}))$ to $L^{q}(\R;X)$;
\item\label{it:exp2} For each $\w>\w'$ there exists a $C\geq 0$ such that one has $\|T(t)x\|_{X}\leq C e^{\w t}\|x\|_{D((-A)^{\alpha})}$ for all $x\in D((-A)^{\alpha})$ and $t\geq0$.
\end{enumerate}
If one of these conditions is satisfied, then \eqref{it:exp1} in fact holds for all $1\leq p\leq q\leq \infty$.
\end{proposition}

As indicated in the introduction, by a Neumann series argument, \eqref{eq:supremumbound3} holds for some $\w'<0$ if $(T(t))_{t\geq0}$ is uniformly bounded and $\sup_{\xi\in\R}\|R(i\xi,A)\|_{\La(X)}<\infty$. Only this setting was considered in Problem \ref{it:problem1}, but the techniques below can be applied to more general exponential decay or growth behavior. In particular, the assumption that $(T(t))_{t\geq0}$ is uniformly bounded does not play a role in this section, and it suffices to assume that $\sup_{\Real(\lambda)\geq0}\|R(\lambda,A)\|_{\La(X)}<\infty$.

The proof of the implication \eqref{it:exp1}$\Rightarrow$\eqref{it:exp2} follows the outline that was sketched in Section \ref{subsec:problem1}, with the following additional remarks. Firstly, since we have not assumed in~Proposition~\ref{prop:exponential} that $(T(t))_{t\geq0}$ is uniformly bounded, in~\eqref{eq:mstability} and~\eqref{eq:deff} one has to choose $\wt{\w}$ such that $(e^{-\wt{\w}t}T(t))_{t\geq0}$ is exponentially stable. More importantly, if \eqref{it:exp1} holds, then the Fourier multiplier with symbol as in \eqref{eq:mstability} is in fact bounded from $L^{p}(\R;D((-A)^{\alpha}))\cap L^{1}(\R;D((-A)^{\alpha}))$ to $L^{\infty}(\R;X)$. Since $f$, as defined in \eqref{eq:deff}, is an element of $L^{p}(\R;D((-A)^{\alpha}))\cap L^{1}(\R;D((-A)^{\alpha}))$,  \eqref{eq:imagemult} then concludes the proof. The implication \eqref{it:exp2}$\Rightarrow$\eqref{it:exp1} follows from \eqref{eq:inverseFourier} and Young's inequality for convolutions, which also yields the final statement.

Proposition \ref{prop:exponential} appears to provide us with a tool to tackle Problem \ref{it:problem1}. However, applications of Proposition \ref{prop:exponential} have mostly focused on trying to establish \eqref{it:exp1} with $p=q$, and in this case the existing Fourier multiplier theorems typically do not yield any useful conclusions for $\alpha\in(0,1)$. On the other hand, the results from the previous section allow us to approach this problem in a different way. 

To be able to apply the results from the previous section, such as Proposition \ref{prop:Fouriertype} and Theorem \ref{thm:typecotype}, we need to obtain sufficient decay of the symbol $\xi\mapsto R(\w+i\xi,A)$ as an operator from $D((-A)^{\alpha})$ to $X$. At first sight, \eqref{eq:supremumbound3} does not seem to provide us with any decay. However, it in fact does, as is a consequence of the following lemma (see e.g.~\cite[Lemma 3.3]{Weis97} or \cite[Proposition 3.4]{Rozendaal-Veraar18a}), the second main tool of this section.

\begin{lemma}\label{lem:decay}
Let $(T(t))_{t\geq0}$ be a $C_{0}$-semigroup with generator $A$ on a Banach space $X$. Let $\w'\in\R$ be such that $\lambda\in\rho(A)$ for all $\lambda\in\C$ with $\Real(\lambda)\geq\w'$, and such that 
\begin{equation}\label{eq:supremum4}
\sup_{\Real(\lambda)\geq\w'}\|R(\lambda,A)\|_{\La(X)}<\infty.
\end{equation}
Then for all $\w\geq \w'$ and $\alpha\in[0,1]$ there exists a $C\geq0$ such that 
\begin{equation}\label{eq:resolventdecay}
\|R(\w+i\xi,A)\|_{\La(D((-A)^{\alpha}),X)}\leq C(1+|\xi|)^{-\alpha}
\end{equation}
for all $\xi\in\R$.
\end{lemma}

In fact, \eqref{eq:supremum4} and \eqref{eq:resolventdecay} are equivalent, but we will only need the implication in Lemma \ref{lem:decay}. A similar statement holds upon replacing the uniform bounds by $R$-bounds. Lemma \ref{lem:decay} can be proved by estimating contour integrals that appear in the Cauchy integral representation of the resolvent.

\subsection{Concrete results on exponential decay}\label{subsec:expresults}

Now we have come to the point where we can combine these tools, yielding the following approach to Problem \ref{it:problem1}:
\begin{enumerate}
\item\label{it:procede1}
Use Lemma \ref{lem:decay} to obtain resolvent decay as in \eqref{eq:resolventdecay} for each $\w\geq \w'$;
\item\label{it:procede2} Apply the Fourier multiplier theorems from Section \ref{sec:multipliers} to the symbol $\xi\mapsto R(\w'+i\xi,A)\in \La(D((-A)^{\alpha}),X)$, for $\alpha\geq 0$ as small as possible;
\item\label{it:procede3} Use Proposition \ref{prop:exponential} to conclude that $\|T(t)x\|_{X}\lesssim e^{\w t}\|x\|_{D((-A)^{\alpha})}$ for all $\w>\w'$, $x\in D((-A)^{\alpha})$ and $t\geq0$.
\end{enumerate}

Note that the value of $\alpha$ in \eqref{it:procede2} depends on the rate of resolvent decay which is required for the theorems in Section \ref{sec:multipliers}. Since the conditions on the symbol in those theorems are in turn dictated by the geometry of $X$ (the fractional domains $D((-A)^{\alpha})$ inherit this geometry), this strategy illustrates the interplay between stability theory, operator-valued Fourier multipliers, and Banach space geometry.

One result which can be obtained in this manner, using Proposition \ref{prop:Fouriertype}, is as follows. 

\begin{theorem}\label{thm:Fourierdecay}
Let $(T(t))_{t\geq0}$ be a $C_{0}$-semigroup with generator $A$ on a Banach space $X$ with Fourier type $p\in[1,2]$. Suppose that $\lambda\in\rho(A)$ for all $\lambda\in\C$ with $\Real(\lambda)\geq0$, and that $\sup_{\Real(\lambda)\geq0}\|R(\lambda,A)\|_{\La(X)}<\infty$. Then there exist $C\geq0$ and $\w<0$ such that 
\begin{equation}\label{eq:Fourierdecay}
\|T(t)x\|_{X}\leq Ce^{-\w t}\|x\|_{D((-A)^{\frac{1}{p}-\frac{1}{p'}})}
\end{equation}
for all $x\in D((-A)^{\frac{1}{p}-\frac{1}{p'}})$ and $t\geq0$.
\end{theorem}

The reader may have observed that, for $p\in(1,2)$, Proposition \ref{prop:Fouriertype} only yields $\|T(t)x\|_{X}\leq Ce^{-\w t}\|x\|_{D((-A)^{\alpha})}$ for $1\geq \alpha>\frac{1}{p}-\frac{1}{p'}$. However, there is a continuity lemma (see~\cite[Lemma 3.5]{Weis97} or \cite[Lemma 5.1]{Rozendaal-Veraar18a}) which then yields \eqref{eq:Fourierdecay}. Moreover, the case where $p=1$ follows from the arguments on general Banach spaces in Section \ref{subsec:problem1}, for $\alpha>1$, combined with the same continuity lemma. 

Theorem \ref{thm:Fourierdecay} was originally obtained in~\cite{Weis-Wrobel96}, and in~\cite{Weis97} it was proved in a similar manner as here but relying on a version of the Mikhlin multiplier theorem on Besov spaces, instead of Proposition~\ref{prop:Fouriertype}. 

By applying Theorem~\ref{thm:typecotype} instead of Proposition \ref{prop:Fouriertype}, one obtains the following variant of Theorem \ref{thm:Fourierdecay}, which was obtained in~\cite{vanNeerven09} using different techniques. 

\begin{theorem}\label{thm:typedecay}
Let $(T(t))_{t\geq0}$ be a $C_{0}$-semigroup with generator $A$ on a Banach space $X$ with type $p\in[1,2]$ and $q\in[2,\infty]$. Suppose that $\lambda\in\rho(A)$ for all $\lambda\in\C$ with $\Real(\lambda)\geq0$, and that $\{R(\lambda,A)\mid \Real(\lambda)\geq0\}$ is $R$-bounded. Then there exist $C\geq0$ and $\w<0$ such that 
\[
\|T(t)x\|_{X}\leq Ce^{-\w t}\|x\|_{D((-A)^{\frac{1}{p}-\frac{1}{q}})}
\]
for all $x\in D((-A)^{\frac{1}{p}-\frac{1}{q}})$ and $t\geq0$.
\end{theorem}

The approach from this section yields additional results, for example for positive semigroups.  
We refer to \cite[Section 5]{Rozendaal-Veraar18a} for details.

\section{Polynomial decay}\label{sec:polynomial}

As we have just seen, $(L^{p},L^{q})$ Fourier multipliers provide one with a unified way to approach exponential stability. However, the connection between Fourier multipliers and the theory of exponential stability, as well as most of the obtained stability results, were previously known. This is no longer the case when considering the more refined decay rates in Problem \ref{it:problem2}. For example, none of the techniques which were described in Section \ref{subsec:problem2} involve a connection between Fourier multiplier theory and polynomial decay rates. 

In this section we will discuss some recent results connecting Fourier multiplier theory to polynomial decay rates. Then, in a similar manner as in the previous section, we will combine this connection with the multiplier theorems from Section \ref{sec:multipliers}, to obtain concrete polynomial decay rates which take into account the geometry of the underlying spaces.

Throughout this section, as in Problem~\ref{it:problem2}, we will consider a $C_{0}$-semigroup $(T(t))_{t\geq0}$ with generator $A$ such that $\lambda\in\rho(A)$ for all $\lambda\in\C$ with $\Real(\lambda)\geq0$. As in the previous section, we do not assume that $(T(t))_{t\geq0}$ is uniformly bounded. Moreover, this time we will \emph{not} restrict ourselves to the case where $\sup_{\xi\in\R}\|R(i\xi,A)\|_{\La(X)}<\infty$. As indicated in Section \ref{subsec:problem2}, damped wave equations give rise to semigroups where one instead has $\|R(i\xi,A)\|_{\La(X)}\eqsim (1+|\xi|)^{\beta}$ for some $\beta>0$ and all $\xi\in\R$, and where one expects polynomial decay rates for orbits $t\mapsto T(t)x$ with $x\in D(A)$. 

\subsection{Polynomial decay and Fourier multipliers}\label{eq:polmult}

 It turns out that polynomial decay behavior can be dealt with in the same manner as exponential decay, although to do so we have to modify Proposition \ref{prop:exponential} and Lemma \ref{lem:decay}. Our extension of Proposition \ref{prop:exponential} is the following special case of \cite[Theorem 4.6]{Rozendaal-Veraar18a}.

\begin{theorem}\label{thm:polynomial}
Let $(T(t))_{t\geq0}$ be a $C_{0}$-semigroup with generator $A$ on a Banach space $X$, and let $\alpha\geq0$ and $n\in\Z_{+}$. Suppose that there exist $\beta,C'\geq0$ such that $\lambda\in\rho(A)$ for all $\lambda\in\C$ with $\Real(\lambda)\geq0$, and $\|R(\lambda,A)\|_{\La(X)}\leq C'(1+|\lambda|)^{\beta}$.
Then the following statements are equivalent:
\begin{enumerate}
\item\label{it:polynomial1} For all $k\in\{1,\ldots, n+1\}$ there exist $p,q\in[1,\infty]$ such that the Fourier multiplier operator with symbol $\xi\mapsto R(i\xi,A)^{k}$ is bounded from $L^{p}(\R;D((-A)^{\alpha}))$ to $L^{q}(\R;X)$;
\item\label{it:polynomial2} There exists a $C\geq0$ such that $\|T(t)x\|_{X}\leq Ct^{-n}\|x\|_{D((-A)^{\alpha})}$ for all $x\in D((-A)^{\alpha})$ and $t\geq1$.
\end{enumerate}
\end{theorem}

Unlike the characterization of exponential stability in Proposition \ref{prop:exponential}, it appears that Theorem \ref{thm:polynomial} was not known prior to \cite{Rozendaal-Veraar18a}. We note that the constants $\beta$ and $C'$ in Theorem \ref{thm:polynomial} play no role in this characterization; we merely need to assume that the resolvent grows at most polynomially.

To prove Theorem~\ref{thm:polynomial} one proceeds in a manner that is analogous to the proof of Proposition \ref{prop:exponential}. For the implication \eqref{it:polynomial1}$\Rightarrow$\eqref{it:polynomial2}, one notes that 
\begin{equation}\label{eq:polynomialFourier}
\int_{0}^{\infty}e^{-i\xi t}t^{n}T(t)x\,\ud t=\frac{1}{(-i)^{n}}\frac{\ud^{n}}{\ud \xi^{n}}\int_{0}^{\infty}e^{-i\xi t}T(t)x\,\ud t=n!(i\xi+A)^{-n-1}x
\end{equation}
for all $\xi\in\R$ and for suitable $x\in X$, as follows from \eqref{eq:inverseFourier}. Moreover, assumption \eqref{it:polynomial1} can be used to show that the Fourier multiplier with symbol $\xi\mapsto R(i\xi,A)^{n+1}$ is bounded from $L^{1}(\R;D((-A)^{\alpha}))\cap L^{p}(\R;D((-A)^{\alpha}))$ to $L^{\infty}(\R;X)$ (in fact, here one only uses the assumption for $k\in\{n-1,n,n+1\}\cap \N$). Then one defines $f$ as in \eqref{eq:deff} and uses the resolvent identity and \eqref{eq:polynomialFourier} to obtain \eqref{it:polynomial2}. The implication \eqref{it:polynomial2}$\Rightarrow$\eqref{it:polynomial1} follows from the fact that operators with uniformly bounded kernels are $(L^{1},L^{\infty})$ multipliers.

Through a rescaling argument, the case $n=0$ in Theorem \ref{thm:polynomial} generalizes Proposition~\ref{prop:exponential}. However, unlike in the case of exponential stability, where $(L^{p},L^{q})$ Fourier multipliers with $p\neq q$ were only used to derive concrete decay estimates from Proposition \ref{prop:exponential}, here the case where $p\neq q$ is even necessary for the \emph{characterization} of polynomial stability using Fourier multipliers. Indeed, it is well known that any bounded $(L^{p},L^{p})$ Fourier multiplier has a uniformly bounded symbol, whereas damped wave equations give rise to uniformly bounded semigroups for which $\|R(i\xi,A)\|_{\La(X)}$ is not uniformly bounded. For such semigroups and for $n=\alpha=0$, \eqref{it:polynomial1} cannot hold if $p=q$, whereas \eqref{it:polynomial2} does hold. On the other hand, if \eqref{it:polynomial1} holds for some $p,q\in[1,\infty]$, then it also holds for $p=1$ and $q=\infty$, again because operators with uniformly bounded kernels are $(L^{1},L^{\infty})$ multipliers.

Now, with Theorem \ref{thm:polynomial} in hand, our next step is to extend Lemma \ref{lem:decay} to this more refined setting, where the resolvent need not be uniformly bounded. The appropriate extension, a special case of \cite[Proposition 3.4]{Rozendaal-Veraar18a}, is as follows.

\begin{proposition}\label{prop:decay2}
Let $(T(t))_{t\geq0}$ be a $C_{0}$-semigroup with generator $A$ on a Banach space $X$. Suppose that there exist $\beta,C'\geq0$ such that $\lambda\in\rho(A)$ for all $\lambda\in\C$ with $\Real(\lambda)\geq0$, and
\begin{equation}\label{eq:blowup}
\|R(\lambda,A)\|_{\La(X)}\leq C'(1+|\lambda|)^{\beta}.
\end{equation}
Then for each $\alpha\in[0,1]$ there exists a $C\geq0$ such that 
\begin{equation}\label{eq:resolventdecay2}
\|R(i\xi,A)\|_{\La(D((-A)^{\alpha+\beta}),X)}\leq C(1+|\xi|)^{-\alpha}
\end{equation}
for all $\xi\in\R$.
\end{proposition}

Again, \eqref{eq:blowup} is in fact equivalent to \eqref{eq:resolventdecay2}, and $R$-bounded versions also hold.

\subsection{Concrete results on polynomial decay}

Now we can follow the same procedure as in Section \ref{subsec:expresults}, under the assumptions of Proposition \ref{prop:decay2}:
\begin{enumerate}
\item\label{it:polprocede1}
Use Proposition \ref{prop:decay2} to obtain resolvent decay as in \eqref{eq:resolventdecay2};
\item\label{it:polprocede2} For a given $n\geq0$, apply the Fourier multiplier theorems from Section \ref{sec:multipliers} to the symbol $\xi\mapsto R(\w'+i\xi,A)^{k}\in\La(D((-A)^{\gamma}),X)$, for all $k\in\{1,\ldots,n+1\}$ and for $\gamma\geq0$ as small as possible;
\item\label{it:polprocede3} Use Theorem \ref{thm:polynomial} to conclude that $\|T(t)x\|_{X}\lesssim t^{-n}\|x\|_{D((-A)^{\gamma})}$ for all $x\in D((-A)^{\gamma})$ and $t\geq1$.
\end{enumerate}

As indicated before, in \eqref{it:polprocede2} one only needs to consider $k\in\{n-1,n,n+1\}\cap \N$, but in many applications this makes no difference, as the term where $k=n+1$ is most problematic. Indeed, to apply the Fourier multiplier theorems from Section \ref{sec:multipliers}, one has to choose $\gamma\geq0$ so that, for a suitable $\alpha\geq0$, one has
\begin{equation}\label{eq:resolventdecay3}
\|R(i\xi,A)^{k}\|_{\La(D((-A)^{\gamma}),X)}\lesssim (1+|\xi|)^{-\alpha}
\end{equation}
for all $1\leq k\leq n+1$ and $\xi\in\R$. Proposition \ref{prop:decay2} yields 
\[
\sup_{\xi\in\R}\|R(i\xi,A)^{k-1}\|_{\La(D((-A)^{(k-1)\beta}),X)}\leq\sup_{\xi\in\R}\|R(i\xi,A)A^{-\beta}\|_{\La(X)}^{k-1}<\infty.
\]
Hence we can apply Proposition \ref{prop:decay2} to the restriction of $A$ to $D((-A)^{(k-1)\beta})$: 
\begin{align*}
&\|R(i\xi,A)^{k}\|_{\La(D((-A)^{k\beta+\alpha}),X)}\\
&\leq \|R(i\xi,A)^{k-1}\|_{\La(D((-A)^{(k-1)\beta}),X)}\|R(i\xi,A)\|_{\La(D((-A)^{k\beta+\alpha}),D((-A)^{(k-1)\beta}))}\\
&\lesssim (1+|\xi|)^{-\alpha}.
\end{align*}
So the term $k=n+1$ dictates the value of $\gamma$ for which one obtains \eqref{eq:resolventdecay3}.

Either way, by combining this reasoning with an interpolation lemma (see \cite[Lemma 4.2]{Rozendaal-Veraar18a}), to deal with the fact that Theorem \ref{thm:polynomial} only involves $n\in\Z_{+}$, one arrives at the following result, Corollary 4.11 in \cite{Rozendaal-Veraar18a}.

\begin{theorem}\label{thm:Hilbertpoly}
Let $(T(t))_{t\geq0}$ be a $C_{0}$-semigroup with generator $A$ on a Hilbert space $X$, and let $s\in[0,\infty)$. Suppose that there exist $\beta,C'\geq0$ such that $\lambda\in\rho(A)$ for all $\lambda\in\C$ with $\Real(\lambda)\geq0$, and $\|R(\lambda,A)\|_{\La(X)}\leq C'(1+|\lambda|)^{\beta}$. Then, for each $s\geq0$, there exists a $C\geq0$ such that 
\[
\|T(t)x\|_{X}\leq Ct^{-s}\|x\|_{D((-A)^{(s+1)\beta})}
\]
for all $x\in D((-A)^{(s+1)\beta})$ and $t\geq1$.
\end{theorem}

Theorem \ref{thm:Hilbertpoly} is sharp up to possible $\veps$-loss, at least for $s=0$. More general results hold under Fourier type, and type and cotype assumptions, but it seems that even Theorem \ref{thm:Hilbertpoly} was not known prior to \cite{Rozendaal-Veraar18a}. A version of Theorem \ref{thm:Hilbertpoly} on general Banach spaces was originally obtained in~\cite{BaEnPrSc06}, and in~\cite{Schnaubelt06} one can find results on polynomial dichotomies under Fourier type assumptions. In \cite{Rozendaal-Veraar18a} one can also find a version of Theorem~\ref{thm:Hilbertpoly} involving polynomial resolvent blow-up at zero, for asymptotically analytic semigroups, a class which includes all analytic semigroups. Although the exponential stability theory for analytic semigroups is well known, this does not seem to be the case for polynomial decay rates, and to the author's best knowledge the analogue of Theorem~\ref{thm:Hilbertpoly} for analytic semigroups and resolvent blow-up at zero was not known prior to \cite{Rozendaal-Veraar18a}.

Theorem~\ref{thm:Hilbertpoly} makes no a priori growth assumptions on the semigroup; it only requires spectral information. It is illustrative to compare this with the Gearhart--Pr\"{u}ss theorem. Let $(T(t))_{t\geq0}$ be a $C_{0}$-semigroup with generator $A$ on a Hilbert space $X$. Suppose that $\lambda\in\rho(A)$ for all $\lambda\in\C$ with $\Real(\lambda)>0$, and that $\sup_{\Real(\lambda)\geq \w'}\|R(\lambda,A)\|_{\La(X)}<\infty$ for each $\w'>0$. Then a rescaled version of Theorem~\ref{thm:Fourierdecay} yields $\|T(t)x\|_{\La(X)}\lesssim e^{\w't}\|x\|_{X}$ for all $\w'>0$, $x\in X$ and $t\geq0$. This also follows from a rescaled version of Theorem~\ref{thm:Hilbertpoly}. On the other hand, if additionally $i\R\subseteq\rho(A)$ and if the resolvent blows up polynomially along the imaginary axis, then Theorem~\ref{thm:Hilbertpoly} specifies fractional domains for which the corresponding semigroup orbits are bounded, or polynomially decaying. Thus the results from this section improve upon those from Section \ref{sec:exponential}, with more refined information.

\section{Refined growth}\label{sec:growth}

In the previous section we improved upon the exponential stability results from Section~\ref{sec:exponential} by obtaining polynomial decay given resolvent blow-up along the imaginary axis. In this section we will extend the results from Section \ref{sec:exponential} in a different manner, by removing the assumption that the imaginary axis is contained in the resolvent set. 
In this case, one should expect general semigroup orbits to display growth behavior, and we will be concerned with obtaining more refined growth rates than the exponential ones from Section~\ref{sec:exponential}. 
It turns out that this problem can again be connected to the theory of $(L^{p},L^{q})$ Fourier multipliers, after which the multiplier theorems from Section \ref{sec:multipliers} can be used to obtain concrete growth rates. 

As in Problem \ref{it:problem3} and Section \ref{subsec:problem3}, let $(T(t))_{t\geq0}$ be a $C_{0}$-semigroup with generator $A$ on a Banach space $X$ such that $\lambda\in\rho(A)$ for all $\lambda\in\C$ with $\Real(\lambda)>0$. Let $g:(0,\infty)\to(0,\infty)$ be non-decreasing and such that
\begin{equation}\label{eq:Kreissgeneral2}
\|R(\lambda,A)\|_{\La(X)}\leq g\Big(\frac{1}{\Real(\lambda)}\Big)
\end{equation}
for all $\lambda\in\C$ with $\Real(\lambda)>0$. 

\subsection{Refined growth and Fourier multipliers}\label{subsec:growthmult}

To relate \eqref{eq:Kreissgeneral2} to growth rates for the associated semigroup, we follow the same procedure as in Sections~\ref{sec:exponential} and~\ref{sec:polynomial}. First, we connect refined growth rates for the semigroup to $(L^{p},L^{q})$ Fourier multiplier properties of the resolvent. Recall from \eqref{eq:multipliernorm} that $\|m\|_{\Ma_{p,q}(X,Y)}=\|T_{m}\|_{\La(L^{p}(\R;X),L^{q}(\R;Y))}$ is the multiplier norm of a symbol $m:\R\to\La(X,Y)$.

\begin{theorem}\label{thm:KreissLpLq}
Let $(T(t))_{t\geq0}$ be a $C_{0}$-semigroup with generator $A$ on a Banach space $X$, such that $\lambda\in\rho(A)$ for all $\lambda\in\C$ with $\Real(\lambda)>0$. Let $\alpha\geq0$, and suppose that there exists a non-decreasing $g:(0,\infty)\to(0,\infty)$, and $p,q\in[1,\infty]$, such that $R(\w+i\cdot,A)\in\Ma_{p,q}(D((1-A)^{\alpha}),X)$ for all $\w\in(0,1]$, with
\begin{equation}\label{eq:subtleLpLq}
\|R(\w+i\cdot,A)\|_{\Ma_{p,q}(D((1-A)^{\alpha}),X)}\leq g\big(\tfrac{1}{\w}\big).
\end{equation}
Then there exists a $C\geq0$ such that 
\begin{equation}\label{eq:subtlegrowth}
\|T(t)x\|_{X}\leq Cg(t)\|x\|_{D((1-A)^{\alpha})}
\end{equation}
for all $x\in D((1-A)^{\alpha})$ and $t\geq1$.
\end{theorem}

Since we do not assume that $0\in\rho(A)$, throughout we have to work with $(1-A)^{\alpha}$ instead of $(-A)^{\alpha}$.

The constant $C$ in \eqref{eq:subtlegrowth} can be made explicit. A version of the converse implication, from \eqref{eq:subtlegrowth} to \eqref{eq:subtleLpLq}, also holds, but the statement is less elegant. We refer the reader to \cite[Theorem 3.11]{Rozendaal-Veraar18b} for the converse implication for polynomial $g$, the proof of which also applies to more general $g$. As before, it relies on Young's inequality. 

We in fact already proved most of Theorem \ref{thm:KreissLpLq} when we proved Proposition~\ref{prop:exponential}, which in turn relied on arguments from Section \ref{subsec:problem1}. Indeed, under the assumptions of the theorem, the Fourier multiplier operator with symbol $\xi\mapsto R(\w+i\xi,A)$ is bounded from $L^{1}(\R;D((1-A)^{\alpha}))\cap L^{p}(\R;D((1-A)^{\alpha}))$ to $L^{\infty}(\R;X)$, with an operator norm which can be estimated in terms of $g(\frac{1}{\w})$. Then~\eqref{eq:imagemult} yields $\|e^{-\w t}T(t)x\|_{X}\lesssim g\big(\tfrac{1}{\w}\big)\|x\|_{D((1-A)^{\alpha})}$ for all $x$ in a suitable dense subspace and for all $t\geq0$. Now set $\w:=1/t$ to arrive at the desired conclusion.

\subsection{Concrete results on refined growth}\label{subsec:growthresults}

With Theorem~\ref{thm:KreissLpLq} in hand, we can follow the same procedure as in the previous sections. One can use \eqref{eq:Kreissgeneral2} to show that the resolvent satisfies the conditions of a suitable $(L^{p},L^{q})$ Fourier multiplier theorem, and then Theorem \ref{thm:KreissLpLq} yields refined growth rates for certain semigroup orbits. On Hilbert spaces this procedure leads to the following result.

\begin{theorem}\label{thm:KreissHilbert}
Let $(T(t))_{t\geq0}$ be a $C_{0}$-semigroup with generator $A$ on a Hilbert space $X$, such that $\lambda\in\rho(A)$ for all $\lambda\in\C$ with $\Real(\lambda)>0$. Suppose that there exists a non-decreasing $g:(0,\infty)\to(0,\infty)$ such that $\|R(\lambda,A)\|_{\La(X)}\leq g\big(\frac{1}{\Real(\lambda)}\big)$ for all such $\lambda$. Then there exists a $C\geq0$ such that $\|T(t)x\|_{X}\leq Cg(t)\|x\|_{X}$ for all $x\in X$ and $t\geq1$.
\end{theorem}

Note that Theorem~\ref{thm:KreissHilbert} recovers the fact that, if the Kreiss condition \eqref{eq:Kreiss} holds, then the associated semigroup grows at most linearly. Since this linear growth is sharp up to possible polynomial $\veps$-loss (see~\cite{EisZwa06}), Theorem~\ref{thm:KreissHilbert} is also essentially sharp. Theorem~\ref{thm:KreissHilbert} was for the most part already contained in~\cite{Helffer-Sjostrand10}, with a better constant but proved without a connection to Fourier multipliers.  An abstract connection to operator-valued $(L^{p},L^{p})$ Fourier multipliers was made in~\cite{Latushkin-Yurov13}, although this connection is of a different nature than Theorem \ref{thm:KreissLpLq}. 

An advantage of the approach in this section, relying on Theorem \ref{thm:KreissLpLq} and $(L^{p},L^{q})$ Fourier multiplier theorems, is that it allows for concrete extensions of Theorem~\ref{thm:KreissHilbert} to more general Banach spaces and semigroups. For example, using the results from Section~\ref{sec:multipliers} and Lemma~\ref{lem:decay}, one obtains an $\alpha\geq0$ such that
\[
\|T(t)x\|_{X}\lesssim g(t)\|x\|_{D((1-A)^{\alpha})}
\]
for all $x\in D((1-A)^{\alpha})$ and $t\geq1$, where $\alpha$ depends on the geometry of $X$. In particular, although on general Banach spaces the semigroup $(T(t))_{t\geq0}$ may grow exponentially even if the Kreiss condition~\eqref{eq:Kreiss} holds, this condition does imply that orbits $t\mapsto T(t)x$ with $x\in D((1-A)^{\alpha})$ grow at most linearly for $\alpha>1$.

In fact, one can show that the conclusion of Theorem~\ref{thm:Hilbertpoly} also holds for positive semigroups on $L^{p}$ spaces or spaces of continuous functions, as well as for asymptotically analytic semigroups. This follows by combining Theorem \ref{thm:KreissLpLq} and Lemma \ref{lem:decay} with $(L^{p},L^{q})$ Fourier multiplier theorems which were not mentioned in Section \ref{sec:multipliers}. We refer to \cite{Rozendaal-Veraar18b} for details.

\section{Refined decay rates}\label{sec:refined}

So far, the common thread in this article has been to use $(L^{p},L^{q})$ Fourier multipliers to derive results about the long-term behavior of semigroups into which one can input information from Banach space geometry. However, we claim that another idea is implicit in the previous sections. Namely, by discerning the connection between the asymptotic theory of evolution equations on the one side, and harmonic analysis in Banach spaces on the other, one is sometimes able to better understand questions about asymptotics of semigroups, even when the relevant harmonic analysis and Banach space geometry are essentially trivial.

In this section we will discuss work on Hilbert spaces regarding Problem \ref{it:problem2}. Here the connection to $(L^{p},L^{q})$ Fourier multiplier theory does not appear explicitly, neither in the results nor in the proofs. Nonetheless, we will attempt to convey how the same paradigm as before plays a key role in the proof of the main result. 

\subsection{The main result}\label{subsec:main}

Throughout, as in Section~\ref{subsec:problem2}, we consider a uniformly bounded $C_{0}$-semigroup $(T(t))_{t\geq0}$ on a Hilbert space $X$, with generator $A$ satisfying $i\R\subseteq\rho(A)$. Until now, none of the results involved a priori growth assumptions on the semigroup; they only relied on spectral information about its generator. By contrast, here the uniform boundedness assumption will play a crucial role. 

Let 
\begin{equation}\label{eq:defM2}
M(\lambda)=\sup_{|\xi|\leq\lambda}\|R(i\xi,A)\|_{\La(X)},\quad\lambda\geq0,
\end{equation}
and suppose that $M(\lambda)\to\infty$ as $\lambda\to\infty$ (if this is not the case, then it follows from Theorem~\ref{thm:Fourierdecay} that $(T(t))_{t\geq0}$ is exponentially stable). To determine the asymptotic behavior of classical solutions to the associated abstract Cauchy problem, i.e.~orbits $t\mapsto T(t)x$ with $x\in D(A)$, we have to obtain decay rates for $\|T(t)A^{-1}\|_{\La(X)}$ as $t\to\infty$. With notation as in Section~\ref{subsec:problem2}, by \cite{Batty-Duyckaerts08}, there exist $c,C>0$ such that
\begin{equation}\label{eq:BattyDuyck2}
\frac{c}{M^{-1}(Ct)}\leq \|T(t)A^{-1}\|_{\La(X)}\leq \frac{C}{M_{\log}^{-1}(ct)}
\end{equation}
for sufficiently large $t$. The right-hand side of \eqref{eq:BattyDuyck2} is slightly larger than the left side, and the goal in this section is to determine when \eqref{eq:BattyDuyck2} can be improved to
\begin{equation}\label{eq:BoriTom}
\frac{c}{M^{-1}(Ct)}\leq \|T(t)A^{-1}\|_{\La(X)}\leq \frac{C}{M^{-1}(ct)}.
\end{equation}
The latter estimate would show, up to constants, exactly what the asymptotic behavior is which can be expected from a general classical solution. Borichev and Tomilov proved that \eqref{eq:BoriTom} holds if $M(\lambda)\eqsim \lambda^{\beta}$ for some $\beta\geq0$ and all $\lambda\geq1$ \cite{Borichev-Tomilov10}. On the other hand, it is known that \eqref{eq:BoriTom} need not hold if $M$ grows too slowly, for example if $M(\lambda)\eqsim \log(\lambda)$ for $\lambda\geq2$ (see~\cite[Example 5.2]{BaChTo16}). Hence it is natural to consider $M$ which do not differ much from a polynomial, as was done in~\cite{BaChTo16}. The relevant definition is as follows.

\begin{definition}\label{def:posinc}
Let $a\geq0$, and let $N:[a,\infty)\to(0,\infty)$ be measurable. For $\alpha\in\R$, the function $N$ is \emph{regularly varying (of index $\alpha$)} if 
\[
\lim_{s\to\infty}\frac{N(\lambda s)}{N(s)}=\lambda^{\alpha}
\]
for all $\lambda\geq1$. Moreover, $N$ has \emph{positive increase} if there exist $\alpha>0$, $c\in(0,1]$ and a positive $s_{0}\geq a$ such that
\begin{equation}\label{eq:alphapos}
\frac{N(\lambda s)}{N(s)}\geq c\lambda^{\alpha}
\end{equation}
for all $\lambda\geq1$ and $s\geq s_{0}$.
\end{definition}

Note that a function of positive increase tends to infinity at a rate that is at least polynomial. Moreover, any regularly varying function of index $\alpha>0$ has positive increase. On the other hand, not all functions of positive increase are regularly varying of positive index, as follows immediately by considering functions which grow faster than any polynomial, such as $N(s)=e^{\alpha s}$ for $\alpha>0$ and $s\geq0$. Functions of tempered growth that oscillate heavily can also have positive increase without being regularly varying of positive index, as is for example the case if $N(s)=s^{\alpha}(2+\sin(s))$ for $\alpha>0$ and $s\geq1$. For the results below, it is relevant to note that the notions differ even when considering non-decreasing functions of tempered growth; consider e.g.~$N(s)=s^{2+m(s)}$ for $m(s)=\sin(\log(\log(s)))$, $s\geq e$.

Asymptotically, regularly varying functions do not deviate much from polynomials. For example, any function of the form $N(s)=s^{\alpha}\log(s)^{\beta}$, for $\alpha,\beta\in\R$, is regularly varying. Functions of the latter kind were studied by Batty, Chill and Tomilov in \cite{BaChTo16}. They proved that, with notation as in \eqref{eq:defM2}, if $M$ satisfies $M(\lambda)\eqsim \lambda^{\alpha}\log(\lambda)^{\beta}$ for some $\alpha>0$ and $\beta\leq0$, then \eqref{eq:BoriTom} holds. They also obtained partial results for $\beta>0$. And they showed that, if $(T(t))_{t\geq0}$ is a normal semigroup, then~\eqref{eq:BoriTom} holds if and only\footnote{In fact, the characterization in \cite[Proposition 5.1]{BaChTo16} involves a condition which, at first sight, differs from \eqref{eq:alphapos}. However, it follows from \cite[Lemma 2.1]{RoSeSt19} that these conditions are equivalent.} if $M$ has positive increase. This brings us to the main result of this section, Theorem 1.1 from \cite{RoSeSt19}.

\begin{theorem}\label{thm:positive}
Let $(T(t))_{t\geq0}$ be a uniformly bounded $C_{0}$-semigroup with generator $A$ on a Hilbert space $X$ such that $i\R\subseteq\rho(A)$. Let $M:[0,\infty)\to(0,\infty)$ be a continuous non-decreasing function such that $\|R(i\xi,A)\|_{\La(X)}\leq M(|\xi|)$ for all $\xi\in\R$. If $M$ has positive increase, then there exist $C,t_{0}\geq 0$ such that
\begin{equation}\label{eq:BorTom2}
\|T(t)A^{-1}\|_{\La(X)}\leq \frac{C}{M^{-1}(t)}
\end{equation}
for all $t\geq t_{0}$.
\end{theorem}

Note that \eqref{eq:BorTom2} differs slightly from the right-hand side of \eqref{eq:BoriTom}, since in \eqref{eq:BorTom2} no constant appears in front of $t$. However, the fact that $M^{-1}(ct)\eqsim M^{-1}(t)$ as $t\to\infty$, for any $c>0$, turns out to characterize the class of functions with positive increase (see \cite[Proposition 2.2]{RoSeSt19}). 

Theorem \ref{thm:positive} applies in particular to any regularly varying function of positive index, such as $M$ satisfying $M(\lambda)\eqsim \lambda^{\alpha}\log(\lambda)^{\beta}$ for $\alpha>0$ and $\beta\in\R$, for which the conclusion is not contained in \cite{BaChTo16} if $\beta>0$. Moreover, since the assumption of positive increase is also necessary for \eqref{eq:BorTom2} to hold if $M$ is as in \eqref{eq:defM2} and if $(T(t))_{t\geq0}$ is a normal semigroup (in fact, it is necessary for a larger class of semigroups), Theorem \ref{thm:positive} completely settles the conjecture of Batty and Duyckaerts, by determining for which $M$ \eqref{eq:BattyDuyck2} can be improved to \eqref{eq:BoriTom}. 

One can obtain improvements over \eqref{eq:BattyDuyck2} for a class of functions which are not of positive increase. Moreover, if $(T(t))_{t\geq0}$ is a normal semigroup, then for all $\veps>0$ one has
\[
\frac{1-\veps}{M_{\max}^{-1}(t)}\leq \|T(t)A^{-1}\|_{\La(X)}\leq \frac{1}{M_{\max}^{-1}(t)}
\]
for sufficiently large $t$, where $M_{\max}(\lambda):=\max_{1\leq s\leq \lambda}M(\lambda/s)\log(s)$ for $\lambda\geq1$. We also note that one encounters semigroup generators with non-polynomial resolvent growth of positive increase when studying wave equations with viscoelastic damping. However, in the remainder we will focus on Theorem~\ref{thm:positive}, and we refer to \cite{RoSeSt19} for these additional results.

\subsection{Ingredients of the proof}\label{subsec:ingredients}

In this subsection we will discuss the following ingredients of the proof of Theorem \ref{thm:positive}: 
\begin{itemize}
\item Uniform boundedness of the semigroup;
\item The semigroup property and Banach space geometry;
\item Tauberian theory.
\end{itemize}

\subsubsection{Uniform boundedness} 

To understand why Theorem \ref{thm:positive} cannot be proved using the same techniques as in Sections~\ref{sec:exponential}, \ref{sec:polynomial} and \ref{sec:growth}, it is illustrative to examine again Theorem~\ref{thm:Hilbertpoly}. That result holds without any a priori assumption on the growth of the semigroup $(T(t))_{t\geq0}$, and for each $s\geq0$ one has 
\begin{equation}\label{eq:Hilbertdecay2}
\|T(t)(-A)^{-(s+1)\beta}\|_{\La(X)}\lesssim t^{-s},\quad t\geq1.
\end{equation}
On the other hand, if we additionally assume that $(T(t))_{t\geq0}$ is uniformly bounded, then one can interpolate \eqref{eq:Hilbertdecay2} with this uniform boundedness to obtain
\begin{equation}\label{eq:BaEnPrSn}
\|T(t)A^{-1}\|_{\La(X)}\lesssim t^{-\frac{1}{\beta}+\frac{1}{(s+1)\beta}},\quad t\geq1.
\end{equation}
The latter estimate improves as $s\to\infty$, but for each $s>0$ it is weaker than \eqref{eq:BattyDuyck2}. In fact, letting $s\to\infty$ in \eqref{eq:BaEnPrSn} yields the main result of~\cite{BaEnPrSc06}, and here one can in fact rely on a weaker version of \eqref{eq:Hilbertdecay2} which also holds on general Banach spaces. On the other hand, Theorem~\ref{thm:Hilbertpoly} is essentially sharp, as follows from an example involving an unbounded semigroup in \cite[Example 4.20]{Rozendaal-Veraar18a}. This means that the assumption that $(T(t))_{t\geq0}$ is uniformly bounded must play a key role in Theorem \ref{thm:positive}, and it cannot be taken into account solely through interpolation with an estimate which also holds for unbounded semigroups.  

Next, let us briefly examine the method of proof that was used in the previous sections. There the crucial connection to Fourier multiplier theory was obtained by applying a Fourier multiplier involving the resolvent to the function $f$ from \eqref{eq:deff}:
\begin{equation}\label{eq:deff2}
f(t):=\begin{cases}
e^{-\wt{w} t}T(t)x&t\geq0,\\
0&t<0,
\end{cases}
\end{equation}
where $\wt{\w}$ is sufficiently large that $(e^{-\wt{\w}t}T(t))_{t\geq0}$ is exponentially stable, and $x\in X$. Since such an $\wt{\w}$ exists for any $C_{0}$-semigroup $(T(t))_{t\geq0}$, we were able to derive statements that depended only on spectral properties of the generator. On the other hand, an inherent drawback of this method is that it does not take into account additional asymptotic information about the semigroup, such as its uniform boundedness.

\subsubsection{Semigroups and Banach space geometry}

A crucial technique for the proof of Theorem \ref{thm:positive} is contained in the proof of \cite[Theorem 4.7]{BaChTo16}. There, for $\tau>0$ fixed, \eqref{eq:deff2} is replaced by a function $f_{\tau}:\R\to X$, defined as follows:
\begin{equation}\label{eq:defftau}
f_{\tau}(t):=\begin{cases}
T(t)x&0\leq t\leq \tau,\\
0&\text{otherwise}.
\end{cases}
\end{equation}
Note that $f_{\tau}$ is only uniformly bounded in $\tau>0$ due to the uniform boundedness assumption on $(T(t))_{t\geq0}$. Hence, albeit in an implicit manner, the uniform boundedness assumption is already embedded in $f_{\tau}$, unlike in~\eqref{eq:deff2}. 

Let $a>0$ and write $f_{\tau,a}(t):=e^{-at}f_{\tau}(t)$ for $t\in\R$. Let $B:D(A)\to X$ commute with $(T(t))_{t\geq0}$ and be such that
\begin{equation}\label{eq:propertyB}
C_{B}:=\sup_{\Real(\lambda)>0}\|BR(\lambda,A)\|_{\La(X)}<\infty.
\end{equation} 
For $\xi\in\R$, set 
\begin{equation}\label{eq:keymult}
m_{a}(\xi):=R(a+i\xi,A)B.
\end{equation} 
Then
\begin{equation}\label{eq:h1}
T_{m_{a}}(f_{\tau,a})(t)=\int_{0}^{\tau}e^{-at}T(t)Be^{a(\tau-t)}f_{\tau}(t)\ud t=\begin{cases}
tBT(t)x&0\leq t\leq \tau,\\
\tau BT(t)x&t>\tau,\\
0&t<0,
\end{cases}
\end{equation}
by \eqref{eq:Laplace}.

From here on we simply write $\|\cdot\|$ for $\|\cdot\|_{X}$. Use \eqref{eq:propertyB} and Plancherel's theorem:
\begin{align*}
\int_{0}^{\tau}e^{-2at}\|tBT(t)x\|^{2}\ud t&\leq \|T_{m_{a}}(f_{\tau,a})\|_{L^{2}(\R;X)}^{2}\leq C_{B}^{2}\|f_{\tau,a}\|_{L^{2}(\R;X)}\\
&=C_{B}^{2}\int_{0}^{\tau}e^{-2at}\|T(t)x\|^{2}\ud t.
\end{align*}
Let $a\to 0$ and use that 
\begin{equation}\label{eq:unifbound}
K:=\sup_{t\geq0}\|T(t)\|_{\La(X)}<\infty,
\end{equation}
to arrive at
\[
\frac{1}{\tau}\int_{0}^{\tau}t^{2}\|BT(t)x\|^{2}\ud t\leq C_{B}^{2}K^{2}\|x\|^{2}.
\]
Now the semigroup property yields
\begin{equation}\label{eq:semigroup}
T(\tau)Bx=\frac{2}{\tau^{2}}\int_{0}^{\tau}tT(\tau-t)T(t)Bx\,\ud t,
\end{equation}
so that one can apply the Cauchy-Schwarz inequality:
\begin{align*}
|\lb \tau T(\tau)Bx,y\rb|&=\Big|\lb \frac{2}{\tau}\int_{0}^{\tau}tT(\tau-t)T(t)Bx\,\ud t,y\rb\Big|=\Big|\frac{2}{\tau}\int_{0}^{\tau}t\lb T(t)Bx,T^{*}(\tau-t)y\rb\ud t\Big|\\
&\leq \Big(\frac{2}{\tau}\int_{0}^{\tau}t^{2}\| T(t)Bx\|^{2}\,\ud t\Big)^{1/2}\Big(\frac{2}{\tau}\int_{0}^{\tau}\|T^{*}(\tau-t)y\|^{2}\,\ud t\Big)^{1/2}\\
&\leq \sqrt{2}K\|y\|\Big(\frac{2}{\tau}\int_{0}^{\tau}t^{2}\| T(t)Bx\|^{2}\,\ud t\Big)^{1/2}\leq 2C_{B}K^{2}\|x\|\|y\|,
\end{align*}
for $y\in X$. Hence $\|T(\tau)Bx\|\leq 2C_{B}K^{2}\|x\|/\tau$. Now, if $M(\lambda)\eqsim \lambda^{\beta}$, then Proposition \ref{prop:decay2} shows that \eqref{eq:propertyB} holds with $B=A^{-\beta}$, and one obtains $\|T(t)A^{-\beta}\|_{\La(X)}\lesssim \tau^{-1}$. Then interpolation with the uniform boundedness of $(T(t))_{t\geq0}$, or suitable extrapolation, recovers the main result of~\cite{Borichev-Tomilov10}: $\|T(\tau)A^{-1}\|_{\La(X)}\lesssim \tau^{-1/\beta}$ as $\tau\to\infty$.

This proof, which is different from the original proof in~\cite{Borichev-Tomilov10}, is relevant for us for three reasons. Firstly, it incorporates the geometry of the underlying Banach space, through Plancherel's theorem. Recall that this is important, because \eqref{eq:BattyDuyck2} is sharp on general Banach spaces. Secondly, it uses the uniform boundedness of the semigroup in a clear manner, both encoded in the function $f_{\tau}$ from \eqref{eq:defftau} and when applying the Cauchy-Schwarz inequality. This is important because \eqref{eq:BoriTom} does not hold for general unbounded semigroups. Finally, it explicitly uses the semigroup property, in \eqref{eq:semigroup}. This is also important, because a version of~\eqref{eq:BattyDuyck2} holds for more general $X$-valued functions, and it is known that the analogue of \eqref{eq:BorTom2} for such functions fails, even on Hilbert spaces.

However, there is one main drawback to this proof. Namely, one has to find $B$ satisfying \eqref{eq:propertyB}. Indeed, one cannot directly apply Plancherel's theorem to the resolvent, which is unbounded by assumption. In the case of polynomial $M$ one can rely on Proposition \ref{prop:decay2}, and then use interpolation to obtain \eqref{eq:BorTom2}, but for more general $M$ it is not clear how to proceed. In~\cite{BaChTo16}, advanced functional calculus tools were developed to deal with this problem for specific $M$. However, it is not known whether such an approach is viable for general $M$ of positive increase.

\subsubsection{Tauberian theory}

By contrast, \eqref{eq:BattyDuyck2} was obtained by Batty and Duyckaerts without any assumptions on $M$. Hence it is natural for us to examine their proof, to determine how the problematic unboundedness of the resolvent is dealt with. In fact, we will discuss a modification by Chill and Seifert of their approach \cite{Chill-Seifert16}. Whereas Batty and Duyckaerts relied on contour integration in~\cite{Batty-Duyckaerts08}, in~\cite{Chill-Seifert16} the Tauberian approach was cast into a framework which is more directly amenable to Fourier analytic techniques.

When applied to the semigroup orbit
\[
g(t):=\begin{cases}
T(t)A^{-1}x&t\geq0,\\
0&t<0,
\end{cases}
\]
for $x\in X$, the proof of \cite[Theorem 2.1]{Chill-Seifert16} uses the decomposition $g=\ph_{R}\ast g+(\delta-\ph_{R})\ast g$ for $R>0$. Here $\delta$ is Dirac mass at zero, and $\ph\in L^{1}(\R)$ is such that $0\leq \wh{\ph}\leq 1$, $\wh{\ph}(\xi)=1$ for $|\xi|\leq 1/2$, and $\wh{\ph}(\xi)=0$ for $|\xi|\geq 1$. Then $\wh{\ph_{R}}(\xi):=\wh{\ph}(\xi/R)$ cuts off to the region $\{\xi\in\Rn\mid|\xi|\leq R\}$.

The strategy is now to write
\begin{equation}\label{eq:decomp}
\|g(t)\|\leq \frac{1}{R}\big(R\|\ph_{R}\ast g\|+R\|(\delta-\ph_{R})\ast g\|\big).
\end{equation}
Then one chooses, for each $t>0$, an $R>0$ such that both of the terms in brackets are bounded independent of $t$ and $R$. As a consequence, $\|g(t)\|$ decays at rate $1/R$. Therefore, to prove the second inequality in \eqref{eq:BattyDuyck2}, one needs to show that the choice $R=M_{\log}^{-1}(ct)$ is allowed, for some $c>0$.
 
For the second term in brackets in \eqref{eq:decomp}, an integration by parts argument yields $\|(1-\ph_{R})\ast g\|\lesssim 1/R$. We will not go into detail here, although it is relevant to note that the resolvent growth rate $M$ does not play a role in this part of the argument. 
 
For the first term in brackets in \eqref{eq:decomp}, it follows from \eqref{eq:Laplace} that
\begin{equation}\label{eq:splitting}
\ph_{R}\ast g(t)=T_{\wh{\ph_{R}}}(g)(t)=\frac{1}{2\pi}\int_{\R}e^{it\xi}\wh{\ph_{R}}(\xi)R(i\xi,A)A^{-1}x\,\ud \xi
\end{equation}
for $t\in\R$. From this point on one integrates by parts successively, using estimates of the form
\begin{equation}\label{eq:absolute}
\int_{\R}\|\wh{\ph_{R}}(\xi)R(i\xi,A)^{k+1}\|_{\La(X)}\ud \xi\lesssim M(R)^{k+1}
\end{equation}
for an additional parameter $k\in\N$. By making a judicial choice of $k$, dependent on $t$ and $M(R)$, and using Stirling's formula, one arrives at the expression
\[
\|g(t)\|\lesssim \frac{1}{R}\big((1+R)^{2}M(R)^{2}e^{-2ct/M(R)}+1\big)
\]
for a suitable $c>0$. Then, for sufficiently large $t$, one may let $R=M_{\log}^{-1}(ct)$.

We briefly discuss some aspects of this proof. Firstly, the term $(1-\ph_{R})\ast g$ corresponds to the high-frequency tail of the resolvent, which involves $R(i\xi,A)$ for $|\xi|\geq R$. This term was problematic in the Hilbert space argument above, because it is not directly accessible via Plancherel's theorem. Here it is taken care of by integrating by parts. 

On the other hand, the penultimate term in \eqref{eq:splitting} contains the low frequencies of $R(i\cdot,A)$. Since $\|\wh{\ph_{R}}(\xi)R(i\xi,A)\|_{\La(X)}\leq M(R)$ by assumption,  Plancherel's theorem might conceivably be applied to this term. In particular, up to a shift by $a$ and the operator $B$, this is the low-frequency part of the symbol in~\eqref{eq:keymult}. Unfortunately, the term $T_{\wh{\ph_{R}}}(g)$ is a Fourier multiplier applied to a semigroup orbit, as opposed to a Fourier multiplier involving the resolvent applied to another function. Hence it is not clear how to apply Plancherel's theorem. Instead, one relies on absolute value estimates as in~\eqref{eq:absolute}.

\subsection{Combining the ingredients}\label{subsec:combine}

We will now indicate how one can combine these ingredients and prove Theorem \ref{thm:positive}.

\subsubsection{Fourier multipliers}

Crude as an absolute value estimate as in~\eqref{eq:absolute} might seem, it in fact brings us back to $(L^{p},L^{q})$ Fourier multiplier theory. Indeed, for any $m\in L^{1}(\R;\La(X))$ and $f\in L^{1}(\Rn;X)$, one has
\begin{align*}
\|T_{m}(f)\|_{L^{\infty}(\R;X)}&\leq \|m\wh{f}\,\|_{L^{1}(\R;X)}\leq \|m\|_{L^{1}(\R;\La(X))}\|\wh{f}\,\|_{L^{\infty}(\R;X)}\\
&\leq \|m\|_{L^{1}(\R;\La(X))}\|f\|_{L^{1}(\R;X)}.
\end{align*}
So each $m\in L^{1}(\R;\La(X))$ is an $(L^{1}(\R;X),L^{\infty}(\R;X))$ Fourier multiplier. Hence, up to the factor $k$, the shift by $a$ and the operator $B$, \eqref{eq:absolute} is in fact an $(L^{1},L^{\infty})$ multiplier estimate for the low-frequency part of the symbol~\eqref{eq:keymult}, which in turn played a key role in \cite{BaChTo16}. To make the final connection to the previous sections, note from Proposition~\ref{prop:Fouriertype} and Theorem~\ref{thm:typecotype} (see also the remark preceding Proposition \ref{prop:decay2}) that on general Banach spaces one can typically only obtain $(L^{1},L^{\infty})$ multiplier estimates. Since \eqref{eq:BattyDuyck2} holds on general Banach spaces, one might indeed expect to encounter $(L^{1},L^{\infty})$ multiplier estimates as in \eqref{eq:absolute}.

With these observations in the back of our mind, we will now give a rough sketch of the proof of Theorem \ref{thm:positive}. 

\begin{proof}[Proof of Theorem \ref{thm:positive}]
Fix $\tau>0$. One again makes use of additional parameters $R$ and $k$, where $R$ will be used to split into low and high frequencies. On the other hand, this time the parameter $k$ is used to introduce a sequence of functions $h_{k}:\R\to X$. In fact, set $h_{0}:=A^{-1}f_{\tau}$, with $f_{\tau}$ as in \eqref{eq:defftau}, and $h_{k}:=T^{*k}\ast h_{0}$ for $k\geq1$. Note that $h_{1}$ already appeared in \eqref{eq:h1}. Moreover, just as was the case in \eqref{eq:semigroup} for $k=1$, the semigroup property yields
\[
T(\tau)A^{-1}x=\frac{k+1}{\tau^{k+1}}\int_{0}^{\tau}T(\tau-t)h_{k}(t)\ud t.
\]
Now write $h_{k}=\ph_{R}\ast h_{k}+(\delta-\ph_{R})\ast h_{k}$, leading to an expression as in \eqref{eq:decomp}, with $g(t)$ replaced by $T(\tau)A^{-1}x$. The high-frequency term, involving $(\delta-\ph_{R})\ast h_{k}$, is dealt with using a similar integration by parts argument as before.

 For the low-frequency term, involving $\ph_{R}\ast h_{k}$, one uses the Cauchy--Schwarz inequality and the uniform boundedness of the semigroup:
\[
\frac{k+1}{\tau^{k+1}}\int_{0}^{\tau}\|T(\tau-t)(\ph_{R}\ast h_{k})(t)\|\ud t\leq K\frac{k+1}{\tau^{k+1/2}}\|\ph_{R}\ast h_{k}\|_{L^{2}(\R;X)},
\]
where $K$ is as in \eqref{eq:unifbound}. Now one can apply Fourier multiplier theory, by observing that $\ph_{R}\ast h_{k}=k!T_{m_{k}}(Ah_{0})$ for 
\[
m_{k}(\xi):=\widehat{\ph_{R}}(\xi)R(i\xi,A)^{k}A^{-1},\quad\xi\in\R.
\]
In passing we note that, with $B=A^{-1}$ and up to the shift by $a$,  $m_{1}$ is the low-frequency part of the symbol in \eqref{eq:Fouriermult}. Either way, since $\|Ah_{0}\|_{L^{2}(\R;X)}\leq K\tau^{1/2}$, we can use that $X$ is a Hilbert space to apply an $(L^{2},L^{2})$ multiplier estimate:
\begin{equation}\label{eq:L2L2}
\frac{k+1}{\tau^{k+1}}\int_{0}^{\tau}\|T(\tau-t)(\ph_{R}\ast h_{k})(t)\|\ud t\leq K^{2}\frac{(k+1)!}{\tau^{k}}\|m_{k}\|_{L^{\infty}(\R;\La(X))}.
\end{equation}
At this point the assumption of positive increase on $M$ comes in. More precisely, in terms of the parameters $\alpha$ and $c$ in \eqref{eq:alphapos}, one can set $k:=\lceil 1/\alpha\rceil$ and use straightforward supremum estimates to obtain
\begin{equation}\label{eq:almost1}
\frac{k+1}{\tau^{k+1}}\int_{0}^{\tau}\|T(\tau-t)(\ph_{R}\ast h_{k})(t)\|\ud t\lesssim \frac{1}{R}(k+1)!\Big(\frac{M(R)}{c\tau}\Big)^{k}\|x\|
\end{equation}
for large $R$. By combining this with the high-frequency estimate, one arrives at
\begin{equation}\label{eq:almost2}
\|T(\tau)A^{-1}x\|\lesssim \frac{1}{R}\Big(\Big(\frac{M(R)}{c\tau}\Big)^{k}+1\Big)\|x\|.
\end{equation}
Note that, in going from \eqref{eq:almost1} to \eqref{eq:almost2}, we used that $k$ is fixed, independent of $\tau$ and $R$, unlike in the proof of \eqref{eq:BattyDuyck2}. This in fact also simplifies the integration by parts procedure for the high-frequency term. 

To conclude the proof one sets $R=M^{-1}(c\tau)$ in \eqref{eq:almost2}, for large enough $\tau$, and uses that $M^{-1}(c\tau)\eqsim M^{-1}(\tau)$ as $\tau\to\infty$, again because $M$ has positive increase.
\end{proof}

We refer to \cite{RoSeSt19} more details on steps which we have left out, such as the integration by parts argument for the high-frequency term. 

Note that the proof involves a combination of techniques and ideas which were already present in \cite{Batty-Duyckaerts08,BaChTo16,Chill-Seifert16}. Note also that Fourier multiplier theory played a crucial role, in terms of an $(L^{2},L^{2})$ multiplier estimate in \eqref{eq:L2L2}. Although in this specific instance the use of $(L^{p},L^{q})$ Fourier multiplier theory is less direct than in the previous sections, this theory provided us with a bridge between disparate techniques, connecting them and leading to a proof of Theorem \ref{thm:positive}.

\appendix

\section{Functional calculus for $C_{0}$-groups}\label{sec:calctypecotype}

The theory of functional calculus deals with mappings $f\mapsto f(A)$, for suitable functions $f$ on subsets of $\C$ and operators $A$ on a Banach space $X$. One typically wants to relate properties of the operator $f(A)$ to those of the function $f$. 

For example, suppose that $-iA$ generates a $C_{0}$-group $(U(t))_{t\in\R}$ on a Banach space $X$, and let $M,\w'\geq0$ be such that $\|U(s)\|_{\La(X)}\leq Me^{\w'|s|}$ for all $s\in\R$. For $\w>\w'$, let $\HT^{\infty}(\St_{\w})$ be the space of bounded holomorphic functions on the strip
\begin{equation}\label{eq:strip}
\St_{\w}:=\{z\in\C\mid \abs{\Imag(z)}<\w\}.
\end{equation}
Let $f\in\HT^{\infty}(\St_{\w})$. Then $f(A)$ can be defined in a natural manner as an unbounded operator on $X$ (see \cite{Haase06a}), and one would like to know when $f(A)$ is in fact bounded. If this is the case for all $f\in\HT^{\infty}(\St_{\w})$, then one says that $A$ has a \emph{bounded $\HT^{\infty}(\St_{\w})$ functional calculus}. This property has applications to maximal regularity (see \cite{Arendt04,Kalton-Weis04,Kunstmann-Weis04}), and it played a key role in the resolution of the Kato square root conjecture \cite{AuHoLaMcITc02}.

It turns out that the results on $(L^{p},L^{q})$ Fourier multipliers from Section \ref{sec:multipliers} also have applications to functional calculus theory. Here we will give such an application, and we indicate how the resulting functional calculus statements can be used to prove stability of numerical approximation schemes for evolution equations. 

For background on the theory of holomorphic functional calculus we refer to \cite{Haase06a}. See also \cite{BaGoTo21a,BaGoTo21b,BaGoTo21c} for recent developments involving a different approach to some of the results mentioned below.

\subsection{Functional calculus and Fourier multipliers}

An effective technique in functional calculus theory involves \emph{transference principles} (see \cite{BeGiMu89,Coifman-Weiss76,Cowling83,Haase09,Haase11}). At the heart of a transference principle lies a commutative diagram of the form
\begin{equation}\label{eq:transferenceCD}
\begin{CD}
F_{1}	@>S_{f}>>	F_{2}\\
@A\iota AA										@VVPV\\
Y			 		@>f(A)>>			Z
\end{CD}
\end{equation}
where $Y$ and $Z$ are subspaces of $X$, $F_{1}$ and $F_{2}$ are function spaces with values in suitable Banach spaces, and $\iota$ and $P$ are bounded. Given such a diagram, $f(A)$ is bounded if $S_{f}$ is. In \cite{Haase09,Haase13,Haase-Rozendaal13} transference principles were applied with $Y=Z=X$ and $F_{1}=F_{2}=L^{p}(\R;X)$ for $p\in(1,\infty)$, while \cite{Haase-Rozendaal16} considered real interpolation spaces $Y=Z=(D(A),X)_{\theta,q}$, with $F_{1}=F_{2}=B^{\theta}_{p,q}(\R;X)$, for $p\in[1,\infty)$, $\theta\in[0,1]$ and $q\in[1,\infty]$. By letting $S_{f}=T_{m_{f}}$ be the Fourier multiplier with a symbol $m_{f}$ which is related to $f$, one then reduces problems in functional calculus theory to questions about operator-valued Fourier multipliers. In particular, one can rely on Plancherel's theorem if $X$ is a Hilbert space, and on versions of Mikhlin's theorem if $p\in(1,\infty)$ and $X$ is a UMD space, or when working on Besov spaces.

When combining transference principles for $C_{0}$-groups with Mikhlin's theorem, as in \cite{Haase09}, one has to restrict to the following function algebra:
\begin{equation}\label{eq:mikhlinalg}
\HT^{\infty}_{1}(\St_{\w}):=\{f\in\HT^{\infty}(\St_{\w})\mid \sup_{z\in\St_{\w}}(1+|z|)|f'(z)|<\infty\},
\end{equation}
endowed with the norm $\|f\|_{\HT^{\infty}_{1}(\St_{\w})}:=\sup_{z\in\St_{\w}}|f(z)|+(1+|z|)|f'(z)|<\infty$. 
If $A$ has a bounded $\HT^{\infty}_{1}(\St_{\w})$ calculus, then the resulting norm bound for $f(A)$ involves $\|f\|_{\HT^{\infty}_{1}(\St_{\w})}$. However, for certain applications, such as the stability of numerical approximation schemes, one requires norm bounds for $f(A)$ in terms of the supremum norm of $f$, for suitable $f\in\HT^{\infty}(\St_{\w})$. It is not clear whether the theory of $(L^{p},L^{p})$ multipliers allows for such applications when $X$ is not a Hilbert space. 

It should be noted that, if $A$ has a bounded $\HT^{1}(\St_{\w})$ calculus, then it also has a bounded $\HT^{\infty}(V_{\ph,\w})$ calculus. Here $V_{\ph,\theta}:=\St_{\w}\cup\Sigma_{\ph}$, for $\ph\in(0,\pi/2)$, is the union of $\St_{\w}$ with a double sector $\Sigma_{\ph}$. However, the associated norm bounds are still not sufficient for certain applications.

Now, the Fourier multiplier results in Section \ref{sec:multipliers} only involve supremum bounds (or $R$-bounds) of a given symbol. In fact, by combining \eqref{eq:Besov} and the Kahane contraction principle, it follows that
\begin{equation}\label{eq:multipliercalc}
T_{m}:B^{\frac{1}{p}-\frac{1}{q}}_{p,1}(\R;X)\to L^{q}(\R;X)
\end{equation}
is bounded for all $m\in L^{\infty}(\R)$, if $X$ has type $p\in[1,2]$ and cotype $q\in[2,\infty]$. 

\subsection{Concrete results on functional calculus}

To apply \eqref{eq:multipliercalc} to functional calculus theory, we note that a weighted version of the group orbit $s\mapsto U(-s)x$ is contained in $B^{\frac{1}{p}-\frac{1}{q}}_{p,1}(\R;X)$, if $q<\infty$ and $x\in (D(A),X)_{\frac{1}{p}-\frac{1}{q},p}$ (see \cite{Haase-Rozendaal16}). Here, as above, $-iA$ generates a $C_{0}$-group $(U(t))_{t\in\R}$. One thus obtains a bounded operator 
\[
\iota:(D(A),X)_{\frac{1}{p}-\frac{1}{q},1}\to B^{\frac{1}{p}-\frac{1}{q}}_{p,1}(\R;X).
\]
 Moreover, there is a bounded projection $P:L^{q}(\R;X)\to X$ such that \eqref{eq:transferenceCD} holds for each $f\in \HT^{\infty}(\St_{\w})$, where $F_{1}=B^{\frac{1}{p}-\frac{1}{q}}_{p,1}(\R;X)$, $F_{2}=L^{q}(\R;X)$, and $S_{f}=T_{m_{f}}$ for a symbol $m_{f}\in L^{\infty}(\Rn)$ satisfying $\|m_{f}\|_{L^{\infty}(\R)}\leq \|f\|_{\HT^{\infty}(\St_{\w})}$.

By combining \eqref{eq:multipliercalc} with this specific instance of \eqref{eq:transferenceCD}, one obtains the following theorem, the main result of \cite{Rozendaal19}.

\begin{theorem}\label{thm:calctype}
Let $-iA$ generate a $C_{0}$-group on a Banach space $X$ with type $p\in[1,2]$ and cotype $[2,\infty)$, and let $M,\w'\geq0$ be such that $\|U(s)\|_{\La(X)}\leq Me^{\w'|s|}$ for all $s\in\R$. Then, for each $\w>\w'$, there exists a $C\geq0$ such that $(D(A),X)_{\frac{1}{p}-\frac{1}{q},1}\subseteq D(f(A))$ and
\[
\|f(A)x\|_{X}\leq C\|x\|_{(D(A),X)_{\frac{1}{p}-\frac{1}{q},1}}
\]
for all $f\in \HT^{\infty}(\St_{\w})$ and $x\in (D(A),X)_{\frac{1}{p}-\frac{1}{q},1}$.
\end{theorem}

In fact, these techniques are not restricted to real interpolation spaces; they can also be applied to domains of fractional powers. More precisely, suppose that $X$ is isomorphic to a complemented subspace of a $p$-convex and $q$-concave Banach lattice, for $p\in[1,2]$ and $q\in[2,\infty)$\footnote{This assumption holds, for example, if $X=L^{r}(\Omega)$ for a measure space $\Omega$ and $r<\infty$, with $p=\min(2,r)$ and $q=\max(2,r)$.}. Then, for each $\lambda>\w$, there exists a $C\geq 0$ such that $D((\lambda+i A)^{-\frac{1}{p}+\frac{1}{q}})\subseteq D(f(A))$ and
\begin{equation}\label{eq:calctype2}
\|f(A)x\|_{X}\leq C\|f\|_{\HT^{\infty}(\St_{\w})}\|(\lambda+i A)^{-\frac{1}{p}+\frac{1}{q}}x\|_{X}
\end{equation}
for all $f\in\HT^{\infty}(\St_{\w})$ and $x\in D((\lambda+i A)^{-\frac{1}{p}+\frac{1}{q}})$.

Recall that the interpolation spaces $\D_{A}(\theta,1)$ and the fractional domains $\D((\lambda+i A)^{-\theta})$ grow as $\theta$ tends to zero, with $X=D((\lambda+iA)^{0})$ corresponding to the case where $X$ is a Hilbert space. Hence Theorem \ref{thm:calctype} and \eqref{eq:calctype2} show that, in the setting under consideration, functional calculus bounds depend in a quantitative manner on how close the geometry of the underlying space is to that of a Hilbert space. This should be compared with the results from the rest of this article, which involve a similar interplay between Fourier multiplier theory and Banach space geometry.

Theorem \ref{thm:calctype} and \eqref{eq:calctype2} have applications to numerical approximation schemes for evolution equations. Here we mention two of these. 

Firstly, an important problem in operator theory is to determine when the \emph{Cayley transform} $(1-A)(1+A)^{-1}$ of an operator $A$ with $-1\in\rho(A)$ is power bounded. This property is in turn equivalent to the stability of the Crank--Nicholson approximation scheme associated with $A$. Theorem \ref{thm:calctype} can be used to derive the stability of this approximation scheme for suitable initial data.

\begin{corollary}\label{cor:cayley}
Let $-A$ generate an exponentially stable $C_{0}$-semigroup $(T(t))_{t\geq 0}\subseteq\La(X)$ on a Banach space $X$ with type $p\in[1,2]$ and cotype $q\in[2,\infty)$. Suppose that $T(t)$ is invertible for all $t\geq0$. Then
\[
\sup_{n\geq1}\|(1-A)^{n}(1+A)^{-n}x\|_{X}<\infty
\]
for all $x\in (D(A),X)_{\frac{1}{p}-\frac{1}{q},1}$.
\end{corollary}

A second application concerns the convergence of rational approximation methods. A rational function $r\in\HT^{\infty}(\C_{-})$ on the complex left half-plane $\C_{-}:=\{z\in\C\mid \Real(z)<0\}$ is \emph{$\mathcal{A}$-stable} if $\|r\|_{\HT^{\infty}(\C_{-})}\leq 1$, and $r$ is a \emph{rational approximation (of the exponential function)} if there exist $k\in\N$ and $C\geq 0$ such that $|r(z)-\ue^{z}|\leq C|z|^{k+1}$ for all $z$ in a complex neighborhood of $0$. Theorem \ref{thm:calctype} implies the following result.

\begin{corollary}\label{cor:approx}
Let $r$ be an $\mathcal{A}$-stable rational approximation, and let $-A$ generate an exponentially stable $C_{0}$-semigroup $(T(t))_{t\geq 0}$ on a Banach space $X$ with type $p\in[1,2]$ and cotype $q\in[2,\infty)$. Suppose that $T(t)$ is invertible for all $t\geq0$. Then $r(-tA/n)^{n}x\to T(t)x$ as $n\to\infty$, for all $x\in (D_{A},X)_{\frac{1}{p}-\frac{1}{q},1}$. 
\end{corollary}

In Corollaries \ref{cor:cayley} and \ref{cor:approx}, under the assumptions of \eqref{eq:calctype2}, one obtains stability and convergence for initial data in the domain of a suitable fractional power of $A$. For similar applications of functional calculus theory to the stability of numerical approximation schemes, we refer to \cite{Egert-Rozendaal13,Gomilko-Tomilov13, Gomilko-Tomilov14,Kovacs04} and \cite[Remark 3.11]{Haase-Rozendaal13}.

\section{Open problems}\label{sec:open}

In this appendix we formulate a few open problems related to the material in this article.

\subsubsection{Endpoint cases for $(L^{p},L^{q})$ Fourier multiplier theorems}

The main goal in Section \ref{sec:multipliers} was to obtain extensions of the multiplier condition \eqref{eq:condm} in the scalar case: 
\[
\sup_{\xi\neq0}|\xi|^{n(\frac{1}{p}-\frac{1}{q})}|m(\xi)|<\infty,
\]
to operator-valued $(L^{p},L^{q})$ multipliers. Proposition \ref{prop:Fouriertype} and Theorem \ref{thm:typecotype} concerned two such extensions, but both results involve an implicit $\veps$ loss with respect to the decay assumption in \eqref{eq:condm}. In Theorem \ref{thm:Besov} we were able to remove this loss for multipliers on Besov spaces, but it remains unclear to what extent the $\veps$ loss can be removed for $(L^{p},L^{q})$ Fourier multipliers.

As motivation for this problem, we note that questions regarding endpoint cases in multiplier theorems have direct applications to stability theory. Indeed, although a continuity lemma yields the critical exponent in the exponential decay result in Theorem \ref{thm:Fourierdecay}, for more refined asymptotic behavior such an approach does not appear to yield concrete statements. Hence the extensions in \cite{Rozendaal-Veraar18a,Rozendaal-Veraar18b} of Theorems \ref{thm:Hilbertpoly} and \ref{thm:KreissHilbert} to non-Hilbertian Banach spaces involve an $\veps$ loss in the fractional domain parameter. Removal of the $\veps$ loss in the relevant Fourier multiplier theorems would yield the endpoint exponents in these stability results.

\subsubsection{Pitt's inequality}

We mention in passing a related problem, concerning vector-valued extensions of Pitt's inequality. In the scalar-valued setting, one has
\begin{equation}\label{eq:Pitt}
\Big(\int_{\Rn}|\wh{f}(\xi)|^{q}|\xi|^{-\gamma q}\ud \xi\Big)^{1/q}\lesssim \Big(\int_{\Rn}|f(x)|^{p}|x|^{\beta p}\ud x\Big)^{1/p},
\end{equation}
for $1<p<q <\infty$, if and only if $\max(0,n(\frac{1}{p}+\frac{1}{q}-1))\leq \gamma<n/q$ and $\beta-\gamma=n(1-\frac{1}{p}-\frac{1}{q})$. This inequality, which describes the mapping properties of the Fourier transform with respect to weighted Lebesgue spaces, simultaneously generalizes the inequalities of Hausdorff--Young and Hardy--Littlewood. 

The problem of determining vector-valued extensions of Pitt's inequality was posed in \cite{Rozendaal-Veraar18}. Substantial progress in this direction was made in \cite{Dominguez-Veraar21}, but, as before, various questions about endpoint cases remain open. Apart from the usefulness of the inequality itself, a better understanding of the endpoint cases in \eqref{eq:Pitt} might lead to a better understanding of the endpoint exponents in $(L^{p},L^{q})$ Fourier multiplier theorems, and vice versa. We refer to \cite{Dominguez-Veraar21} for more details.

\subsubsection{Exponential decay without $R$-boundedness}

Theorem \ref{thm:Fourierdecay} shows that, under the assumption that $X$ has Fourier type $p$ and that the resolvent of $A$ is  uniformly bounded on the standard right half-plane, $\|T(t)x\|_{X}$ decays exponentially for all $x\in D((-A)^{\frac{1}{p}-\frac{1}{p'}})$. On the other hand, by Theorem \ref{thm:typedecay}, the same decay statement holds for all $x\in D((-A)^{\frac{1}{p}-\frac{1}{q}})$ if $X$ has type $p$ and cotype $q$, and if the resolvent is $R$-bounded on the standard right half-plane. Since every Banach space with Fourier type $p$ has type (at least) $p$ and cotype (at most) $p'$, and because many spaces have better type and cotype properties than is dictated by their Fourier type, the conclusion of Theorem \ref{thm:typedecay} improves upon that of Theorem \ref{thm:Fourierdecay}. However, the $R$-boundedness assumption on the resolvent in Theorem \ref{thm:typedecay} is stronger than the uniform boundedness assumption in Theorem \ref{thm:Fourierdecay}, and it is an open problem to determine whether this $R$-boundedness assumption is necessary.

\subsubsection{The Kreiss condition on Hilbert spaces}

On Hilbert spaces, the Kreiss condition \eqref{eq:Kreiss} guarantees that the associated semigroup grows at most linearly. Moreover, for each $\gamma\in[0,1)$ there exists a semigroup $(T(t))_{t\geq0}$ on a Hilbert space $X$ such that $\|T(t)\|_{\La(X)}\gtrsim t^{\gamma}$ for $t\geq1$, and such that its generator satisfies the Kreiss condition \cite{EisZwa06}. It is an open problem whether such an example exists for $\gamma=1$, or whether the Kreiss condition implies a slightly sharper growth bound.

\subsubsection{Decay rates for functions of quasi-positive increase}

As indicated in Section \ref{sec:refined} and with notation as in Theorem \ref{thm:positive}, to obtain the inequality
\[
\|T(t)A^{-1}\|_{\La(X)}\leq \frac{C}{M^{-1}(t)}
\]
it is both necessary and sufficient to assume that $M$ has positive increase. However, there are various slowly-growing functions which do not have the property of positive increase, and it is natural to ask whether one can improve upon the Batty--Duyckaerts estimate \eqref{eq:BattyDuyck2} for some of these functions. In \cite{RoSeSt19} the class of functions of \emph{quasi-positive increase} was introduced, and improvements over \eqref{eq:BattyDuyck2} were obtained for this class. However, it is unclear what the optimal decay rate is for certain interesting functions of quasi-positive increase, such as $M(\lambda)=(\log(\lambda))^{\alpha}$ for $\lambda\geq e$ and $\alpha>0$.

\subsubsection{Refined decay rates and Banach space geometry}

Although the theory of $(L^{p},L^{q})$ Fourier multipliers provides one with a viewpoint to connect ideas from \cite{Batty-Duyckaerts08,Chill-Seifert16,BaChTo16} and prove Theorem \ref{thm:positive}, it is unclear whether such multipliers can be used to obtain improvements of the Batty--Duyckaerts estimate \eqref{eq:BattyDuyck2} on certain non-Hilbertian Banach spaces. A better understanding of the role that $(L^{p},L^{q})$ Fourier multipliers play in this setting might in turn lead to an improved understanding of the Hilbert space case and the remaining open problems there.

\subsubsection{Functional calculus for semigroup generators}

In Appendix \ref{sec:calctypecotype} we applied $(L^{p},L^{q})$ Fourier multiplier theorems to obtain functional calculus results for $C_{0}$-groups. Now, it is known that the functional calculus theory for $C_{0}$-semigroups is considerably more subtle than the theory for $C_{0}$-groups (see e.g.~\cite{BaGoTo21a,BaGoTo21b}). Nonetheless, in \cite{Haase11,Haase-Rozendaal13} transference principles were combined with Fourier multiplier theorems to obtain functional calculus results for $C_{0}$-semigroups. In those cases $(L^{p},L^{p})$ Fourier multiplier theorems were used. It is thus a natural problem to determine to what extent $(L^{p},L^{q})$ Fourier multiplier theorems, for $p\neq q$, can be combined with transference principles for semigroups to derive functional calculus properties of $C_{0}$-semigroups on non-Hilbertian Banach spaces.

\subsubsection{Other applications of $(L^{p},L^{q})$ Fourier multipliers}

It is perhaps fitting to conclude with a somewhat more general open problem. Namely, one of the aims in preparing this article has been to convey how the theory of $(L^{p},L^{q})$ Fourier multipliers can be applied to problems in a variety of other settings. It is not clear to the author why such applications should be restricted to stability theory and functional calculus theory. Hence the work discussed here naturally leads to the problem of determining additional fruitful applications of the theory of $(L^{p},L^{q})$ Fourier multipliers.

\bibliographystyle{plain}
\bibliography{Bibliography}

\end{document}